\documentclass[12pt,reqno]{amsart}
\usepackage{fullpage}
\usepackage{color}
\usepackage{amsmath}
\usepackage{amssymb}
\usepackage{bm}

\newtheorem{theorem}{Theorem}[section]
\newtheorem{lemma}[theorem]{Lemma}
\newtheorem{prop}[theorem]{Proposition}

\theoremstyle{definition}
\newtheorem{definition}[theorem]{Definition}

\newtheorem*{remark}{Remark}

\renewcommand{\mod}[1]{{\ifmmode\text{\rm\ (mod~$#1$)}\else\discretionary{}{}{\hbox{ }}\rm(mod~$#1$)\fi}}
\newcommand{\ep}{\varepsilon}

\newcommand{\gbinom}[2]{\genfrac{[}{]}{0pt}{}{#1}{#2}}
\newcommand{\ovomega}{{\overline\omega}}
\newcommand{\cov}{\mathop{\rm cov}\nolimits}
\newcommand{\li}{\mathop{\rm li}\nolimits}

\newcommand{\F}{{\mathbb F}}

\newcommand{\R}{{\mathbb R}}

\newcommand{\Z}{{\mathbb Z}}
\newcommand{\Znt}{{\Z_n^\times}}
\newcommand{\ba}{{\bm a}}
\newcommand{\bb}{{\bm b}}
\newcommand{\balpha}{{\bm\alpha}}
\newcommand{\bbeta}{{\bm\beta}}
\newcommand{\bgamma}{{\bm\gamma}}
\begin{document}

\title{The distribution of the number of subgroups of the multiplicative group}
\author[Greg Martin and Lee Troupe]{Greg Martin and Lee Troupe}
\address{Department of Mathematics \\ University of British Columbia \\ Room 121, 1984 Mathematics Road \\ Canada V6T 1Z2}
\email{gerg@math.ubc.ca}
\email{ltroupe@math.ubc.ca}
\maketitle

\begin{abstract}
Let $I(n)$ denote the number of isomorphism classes of subgroups of $(\Z/n\Z)^\times$, and let $G(n)$ denote the number of subgroups of $(\Z/n\Z)^\times$ counted as sets (not up to isomorphism). We prove that both $\log G(n)$ and $\log I(n)$ satisfy Erd{\H o}s--Kac laws, in that suitable normalizations of them are normally distributed in the limit. Of note is that $\log G(n)$ is not an additive function but is closely related to the sum of squares of additive functions. We also establish the orders of magnitude of the maximal orders of $\log G(n)$ and $\log I(n)$.
\end{abstract}

\section{Introduction}

The distribution of values of additive functions has long been of interest to number theorists. Perhaps the most famous result in this area is the celebrated Erd{\H o}s--Kac theorem: if $\omega(n)$ and $\Omega(n)$ denote, respectively, the number of distinct prime factors of $n$ and the number of prime factors of $n$ counted with multiplicity, then the distributions of the values of both
\[
\frac{\omega(n) - \log\log n}{\sqrt{\log\log n}} \quad\text{and}\quad \frac{\Omega(n) - \log\log n}{\sqrt{\log\log n}}
\]
tend to the standard normal distribution. In other words, both $\omega(n)$ and $\Omega(n)$ are, in the limit, ``normally distributed with mean $\log\log n$ and variance $\log\log n$''. Indeed, Erd{\H o}s and Kac~\cite{ek40} established this property for a large class of additive functions, and many subsequent authors have widened even further the set of functions for which we know such Erd{\H o}s--Kac laws. In this paper, we establish Erd{\H o}s--Kac laws for two functions that count subgroups of a natural family of finite abelian groups, as we now describe.

Let $\Znt = (\Z/n\Z)^\times$ denote the multiplicative group of units modulo~$n$.
Let $G(n)$ denote the number of subgroups of $\Znt$, counted as sets (rather than up to isomorphism), so that $G(8) = 5$, for example. The function $G(n)$ is not a multiplicative function of $n$, but it does have the property that $G(n) = \prod_{p \mid \phi(n)} G_p(n)$ (as we shall see below), where $G_p(n)$ denotes the number of $p$-subgroups of $\Znt$. One could perhaps say that $G(n)$ is ``a multiplicative function of $\phi(n)$'', or simply ``$\phi$-multiplicative,'' making
\begin{align}\label{phi additive}
\log G(n) = \sum_{p \mid \phi(n)} \log G_p(n)
\end{align}
a ``$\phi$-additive'' function. Our primary aim is to show that $\log G(n)$ possesses enough structure to satisfy a similar Erd{\H o}s--Kac law:

\begin{theorem}\label{main theorem}
Define
\[
A_0 = \frac{1}{4} \sum_{p} \frac{p^2 \log p}{(p - 1)^3(p + 1)} \quad\text{and}\quad A = A_0 + \frac{\log2}2 \approx 0.72109
\]
and
\[
B = \frac14 \sum_p \frac{p^3(p^4 - p^3 - p^3 - p - 1)(\log p)^2}{(p - 1)^6(p + 1)^2(p^2 + p + 1)}, 
\]
and set $C = \frac{(\log 2)^2}{3} + 2 A_0 \log 2 + 4A_0^2 + B \approx 3.924$. (Both sums are taken over all primes~$p$.) Then for every real number $u$,
\begin{equation*}
\lim_{x \to \infty} \frac{1}{x} \#\big\{n \leq x \colon \log G(n) < A(\log\log n)^2 + u\cdot \sqrt C(\log\log n)^{3/2} \big\} = \frac{1}{\sqrt{2\pi}} \int_{-\infty}^u e^{-t^2/2} \, dt.
\end{equation*}
In other words, the quantity $\log G(n)$ is normally distributed, with mean $A (\log\log n)^2$ and variance $C(\log\log n)^3$.
\end{theorem}

We briefly indicate the overall structure of the proof of Theorem 1.1. First, we understand the typical values of $\log G(n)$ by writing them as a linear combination of squares of well-understood additive functions, together with one anomalous ``$\phi$-additive'' function.

\begin{prop}\label{log G as a polynomial intro version}
Set $X = (\log\log x)^{1/2}(\log\log\log x)^{2}$. For any positive integer $n$, define
\begin{equation}  \label{Qnx definition}
P_n(x) = \log 2 \cdot \omega(\phi(n)) + \frac{1}{4} \sum_{q \leq X} \omega_q(n)^2 \Lambda(q),
\end{equation}
where $\Lambda(q)$ denotes the usual von Mangoldt function, and where the $\omega_q$ are additive functions defined in Definition~\ref{omega q def} below.
Then for all but $O(x/\log\log\log x)$ integers $n\le x$,
\begin{equation} \label{log gn as a polynomial intro version}
\log G(n) = P_n(x) + O\bigg(\frac{(\log\log x)^{3/2}}{\log\log\log x}\bigg).
\end{equation}
\end{prop}

For any function $f(n)$, define the ``mean''
\begin{equation}  \label{mu definition}
\mu(f) = \mu(f;x) = \sum_{p \leq x} \frac{f(p)}{p},
\end{equation}
and set
\begin{equation}  \label{Dx definition}
D(x) = \log 2 \cdot \mu(\omega \circ \phi) + \frac{1}{4} \sum_{q \leq X} \mu(\omega_q)^2 \Lambda(q)
\end{equation}
(so that $D(x)$ is simply $P_n(x)$ with each function of $n$ replaced by its mean).
Our strategy is to show that the values of $P_n(x)$ for $n \leq x$ are, asymptotically as $x$ tends to~$\infty$, normally distributed with mean $D(x)$ and variance $C(\log\log x)^{3}$, with $C$ defined as in Theorem \ref{main theorem}. We carry out this strategy via the ``method of moments.''

\begin{prop}\label{moment asymptotics}
For any positive integer $h$, define the ``$h$th moment''
\begin{equation}  \label{Mh def}
M_h(x) = \sum_{n \leq x} \big( P_n(x) - D(x) \big)^h.
\end{equation}
Then
\[
\lim_{x \to \infty} \frac{M_h(x)}{C^{h/2}x(\log\log x)^{3h/2}} = \begin{cases} \frac{h!}{(h/2)!2^{h/2}}, &\mbox{if } h \text{ is even,} \\ 
0, & \mbox{if } h \text{ is odd.} \end{cases}
\]
\end{prop}

\noindent The quantity $\frac{h!}{(h/2)!2^{h/2}}$ for even $h$ is precisely the $h$th moment of the standard normal distribution, and it is a famous lemma of Chebyshev that the normal distribution is determined by its moments (see Section~\ref{flourish} for more details).

Our proof is inspired by work of Granville and Soundararajan~\cite{gs07}, who described a way to organize method-of-moments proofs in number theory to make the main terms more readily identifiable. The proof herein is tailored to the specific function $P_n(x)$ mentioned above, which can be viewed as a quadratic polynomial (in increasingly many variables) being evaluated at values of specific additive functions. For any fixed polynomial, one can apply the same techniques to its evaluation at values of additive functions from a much more general class, thereby obtaining Erd{\H o}s--Kac laws for these polynomials of additive functions as well (including, for example, Erd{\H o}s--Kac laws for products of additive functions). This generalization is the subject of forthcoming work by the authors.

Since $G(n)$ counts subgroups of $\Znt$ as sets, the reader might wonder about the equally natural function $I(n)$ that counts subgroups of $\Znt$ up to isomorphism. It turns out to be much easier to establish an Erd{\H o}s--Kac law for $\log I(n)$, partially because $I(n)$ is a $\phi$-multiplicative function of a much simpler type, but mostly because we can leverage existing work of Erd{\H o}s and Pomerance~\cite{ep85} on the number of prime factors of $\phi(n)$ to greatly shorten our proof.

\begin{theorem}  \label{logIn EK thm}
For every real number $u$,
\begin{multline*}
\lim_{x \to \infty} \frac{1}{x} \#\bigg\{n \leq x \colon \log I(n) < \frac{\log 2}{2}(\log\log n)^2 + u\cdot \sqrt{\frac{\log 2}{3}} (\log\log n)^{3/2} \bigg\} \\
= \frac{1}{\sqrt{2\pi}} \int_{-\infty}^u e^{-t^2/2} \, dt.
\end{multline*}
In other words, the quantity $\log I(n)$ is normally distributed, with mean $\frac{\log 2}{2} (\log\log n)^2$ and variance $\frac{\log 2}{3}(\log\log n)^3$.
\end{theorem}

\noindent The leading constant here, $\frac{\log2}2 \approx 0.34657$, for the typical size of $\log I(n)$ is a bit less than half the leading constant $A$ for the typical size of $\log G(n)$ in Theorem~\ref{main theorem}; in other words, the total number $G(n)$ of subgroups of $\Znt$ is typically a bit more than the square of the number $I(n)$ of isomorphism classes of subgroups of~$\Znt$.

We begin by establishing Proposition~\ref{log G as a polynomial intro version} in Section~\ref{from partitions section}, which will require a brief digression into counting subgroups of finite abelian $p$-groups using partitions and Gaussian binomial coefficients. Sections~\ref{notation and setup} through~\ref{flourish} comprise the proof of Theorem~\ref{main theorem}, with the verification of Proposition~\ref{moment asymptotics} taking place in Section~\ref{proof of moment asymptotics}; a more detailed roadmap is provided in Section~\ref{notation and setup}, along with notation and conventions that will be used through the rest of the paper. Finally, Section~\ref{max order section} contains the proof of the aforementioned theorem about $I(n)$, along with proofs of the following maximal-order results for $\log G(n)$ and $\log I(n)$:


\begin{theorem}\label{max order Gn}
The order of magnitude of the maximal order of $\log G(n)$ is ${(\log x)^2}/{\log\log x}$. More precisely,
\[
\frac1{16} \frac{(\log x)^2}{\log\log x} + O\bigg( \frac{(\log x)^2\log\log\log x}{(\log\log x)^2} \bigg) \leq \max_{n \leq x} \big( \log I(n) \big) \leq \frac1{4} \frac{(\log x)^2}{\log\log x} + O\bigg( \frac{(\log x)^2}{(\log\log x)^2} \bigg).
\]
\end{theorem}

\begin{theorem}\label{max order In}
The order of magnitude of the maximal order of $\log I(n)$ is ${\log x}/{\log\log x}$. More precisely,
\[
\frac{\log 2}{5} \frac{\log x}{\log\log x} + O\bigg( \frac{\log x}{(\log\log x)^2} \bigg) \leq \max_{n \leq x} \big( \log I(n) \big) \leq \pi\sqrt{\frac23} \frac{\log x}{\log\log x} + O\bigg( \frac{\log x}{(\log\log x)^2} \bigg).
\]
\end{theorem}

\section{Expressing $\log G(n)$ as a polynomial of additive functions}  \label{from partitions section}

In this section we prove Proposition~\ref{log G as a polynomial intro version}. First, we import a classical identity for the number of subgroups of a finite abelian $p$-group, which we alter into an approximate form that is suitable for our application. Then we describe exactly the $p$-Sylow subgroup of the multiplicative group $\Znt$ and record its approximate number of subgroups. Finally we sum this contribution over all primes $p$, which mostly involves dealing with the complication of truncating this sum suitably to avoid being overwhelmed with error terms; we employ some ``anatomy of integers'' arguments to show that this truncation is valid for almost all integers~$n$.

\subsection{Subgroups of $p$-groups}  \label{subgroups of p groups section}

Let us recall, from the classification of finitely generated abelian groups, that every finite abelian group of size $p^m$ can be uniquely written in the form $\Z_{p^{\alpha_1}} \times \Z_{p^{\alpha_2}} \times \cdots$ for some nonincreasing sequence $(\alpha_1,\alpha_2,\dots)$ of nonnegative integers summing to $m$. (We avoid naming the length of such sequences by the convention that all but finitely many of the $\alpha_j$ equal~$0$.) In other words, isomorphism classes of finite abelian $p$-groups are in one-to-one correspondence with partitions $\balpha = (\alpha_1,\alpha_2,\dots)$ of~$m$.

A subpartition $\bbeta$ of a partition $\balpha$ is a nonincreasing sequence $(\beta_1,\beta_2,\dots)$ of positive integers such that $\beta_j \le \alpha_j$ for all $j\ge1$; we write $\bbeta \preceq \balpha$ when $\bbeta$ is a subpartition of~$\balpha$. It is easy (though not quite trivial) to see that $\Z_{p^{\alpha_1}} \times \Z_{p^{\alpha_2}} \times \cdots$ contains an isomorphic copy of $\Z_{p^{\beta_1}} \times \Z_{p^{\beta_2}} \times \cdots$ if and only if $\bbeta \preceq \balpha$. We are interested in more precise information, however, about the number of subgroups of $\Z_{p^{\alpha_1}} \times \Z_{p^{\alpha_2}} \times \cdots$ that are isomorphic to $\Z_{p^{\beta_1}} \times \Z_{p^{\beta_2}} \times \cdots$.

\begin{definition}
Given partitions $\bbeta \preceq \balpha$ and a prime $p$, define $N_p(\balpha,\bbeta)$ to be the number of subgroups inside $\Z_{p^{\alpha_1}} \times \Z_{p^{\alpha_2}} \times \cdots$ that are isomorphic to $\Z_{p^{\beta_1}} \times \Z_{p^{\beta_2}} \times \cdots$. Define $N_p(\balpha)$ to be the number of subgroups inside $\Z_{p^{\alpha_1}} \times \Z_{p^{\alpha_2}} \times \cdots$ (as sets, not up to isomorphism), so that $N_p(\balpha) = \sum_{\bbeta \preceq \balpha} N_p(\balpha,\bbeta)$.
\end{definition}

As it happens, there is a classical formula for $N_p(\balpha,\bbeta)$, most conveniently expressed in terms of conjugate partitions.
Every partition $\balpha$ has a conjugate partition $\ba$, which is most easily obtained by transposing the Ferrers diagram corresponding to~$\balpha$. The number of parts (nonzero elements) of the conjugate partition $\ba$ is exactly equal to $\alpha_1$, and in general $\alpha_j$ equals the number of parts of $\ba$ that are at least $j$ in size; by the same token, the first part $a_1$ of $\ba$ is equal to the number of parts of $\balpha$, and so on.

We quote this classical formula, which can be found in~\cite[equation (1)]{ste92} and the references cited therein:

\begin{lemma}
\label{quoting number of subgroups lemma}
Let $p$ be prime, and let $\bbeta \preceq \balpha$ be partitions. Let $\ba = (a_1,a_2,\dots,a_{\alpha_1},0,\dots)$ and $\bb = (b_1,b_2,\dots,b_{\beta_1},0,\dots)$ be the conjugate partitions to $\balpha$ and $\bbeta$, respectively. Then
\[
N_p(\balpha,\bbeta) = \prod_{j=1}^{\alpha_1} p^{(a_j-b_j)b_{j+1}} \gbinom{a_j-b_{j+1}}{b_j-b_{j+1}}_p.
\]
Here, $\gbinom k\ell_p$ is the Gaussian binomial coefficient, defined to be $0$ if $\ell<0$ or $\ell>k$, and otherwise defined by
\begin{equation}  \label{Gauss binom def}
\gbinom k\ell_p = \prod_{j=1}^\ell \frac{p^{k-\ell+j}-1}{p^j-1}.
\end{equation}
\end{lemma}

The reader might gain some intuition from considering the case where $\balpha=(1,\dots,1,0,\dots)$ and $\bbeta=(1,\dots,1,0,\dots)$ are the finest possible partitions of $k$ and $\ell$, respectively, so that $N_p(\balpha,\bbeta)$ is the number of $\ell$-dimensional subspaces of $\F_p^k$. In this case, $\ba=(k,0,\dots)$ and $\bb=(\ell,0,\dots)$ and so $N_p(\balpha,\bbeta)$ is simply $\gbinom k\ell_p$. It can be seen that the numerator of the formula~\eqref{Gauss binom def} is, up to a power of $p$, the number of $k\times\ell$ matrices over $\F_p$ with full rank $\ell$ (and the column space of each such matrix defines an $\ell$-dimensional subspace of $\F_p^k$), while the denominator is, up to the same power of $p$, the number of invertible $\ell\times\ell$ matrices over $\F_p$ (which act by left multiplication on the set of $k\times\ell$ matrices while preserving their column spaces).

We will prefer an approximate version of the formula from Lemma~\ref{quoting number of subgroups lemma}, which the following pair of lemmas provides.

\begin{lemma}
\label{gaussian binomial estimate lemma}
For any prime $p$ and any integers $0\le \ell\le k$, there exists a real number $0\le \theta < 6$ such that
\[
\gbinom k\ell_p = p^{\ell(k-\ell)} (1+\theta p^{-1}).
\]
\end{lemma}

\begin{proof}
For any integers $k\ge \ell\ge j\ge1$, we have the inequalities
\begin{equation}
\label{squeeze between}
p^{k-\ell} \le \frac{p^{k-\ell+j}-1}{p^j-1} \le \frac{p^{k-\ell+j}}{p^j-1} = \frac{p^{k-\ell}}{1-p^{-j}}.
\end{equation}
Note that
\[
\prod_{j=2}^\ell (1-p^{-j}) \ge 1-\sum_{j=2}^\ell p^{-j} > 1-\frac1{p(p-1)},
\]
and so
\[
\prod_{j=1}^\ell \frac1 {1-p^{-j}} < \frac1{1-p^{-1}} \bigg( 1-\frac1{p(p-1)} \bigg)^{-1} \le 1+6p^{-1},
\]
where the last inequality follows by a simple calculation. Therefore equation~\eqref{squeeze between} implies that
\begin{equation*}
p^{\ell(k-\ell)} \le \prod_{j=1}^\ell \frac{p^{k-\ell+j}-1}{p^j-1} \le p^{\ell(k-\ell)} \prod_{j=1}^\ell \frac1 {1-p^{-j}} < p^{\ell(k-\ell)} (1+6p^{-1}),
\end{equation*}
which establishes the lemma.
\end{proof}

\begin{lemma}
\label{subgroups fixed type lemma}
Given partitions $\bbeta \preceq \balpha$, let $\bb$ and $\ba$ be the partitions conjugate to $\bbeta$ and $\balpha$ respectively. For any prime $p$, 
\[
N_p(\balpha,\bbeta) = \bigg( \prod_{j=1}^{\alpha_1} p^{(a_j-b_j)b_j} \bigg) \big( 1+\theta p^{-1} \big)^{\alpha_1}
\]
for some real number $0\le \theta < 6$.
\end{lemma}

\begin{proof}
By Lemmas~\ref{quoting number of subgroups lemma} and~\ref{gaussian binomial estimate lemma}, there exist real numbers $0\le\theta_j<6$ such that
\begin{align*}
N_p(\balpha,\bbeta) &= \prod_{j=1}^{\alpha_1} p^{(a_j-b_j)b_{j+1}} \gbinom{a_j-b_{j+1}}{b_j-b_{j+1}}_p \\
&= \prod_{j=1}^{\alpha_1} p^{(a_j-b_j)b_{j+1}} \big( p^{(b_j-b_{j+1}) (a_j-b_j)} (1+\theta_jp^{-1}) \big) \\
&= \prod_{j=1}^{\alpha_1} p^{(a_j-b_j)b_j} (1+\theta_jp^{-1}).
\end{align*}
The lemma now follows from the intermediate value property of the continuous function $f(\theta) = ( 1+\theta p^{-1})^{\alpha_1}$ on the interval $0\le\theta\le6$, along with the observation that
\[
f(0) \prod_{j=1}^{\alpha_1} p^{(a_j-b_j)b_j} \le \prod_{j=1}^{\alpha_1} p^{(a_j-b_j)b_j} (1+\theta_jp^{-1}) \le f(6) \prod_{j=1}^{\alpha_1} p^{(a_j-b_j)b_j}.
\]
\end{proof}

Finally, we want to sum $N_p(\balpha,\bbeta)$ over all subpartitions $\bbeta$ of $\balpha$. It turns out that the dominant contribution to this sum comes from the subpartitions $\bbeta$ nearest to $\frac12\balpha$.

\begin{lemma}
\label{odd even sum lemma}
For any integer $a\ge0$ and any prime $p$, we have
\[
\sum_{b=0}^a p^{(a-b)b} = p^{a^2/4+O(1)} \quad\text{and}\quad p^{(a-\lfloor \frac a2\rfloor)\lfloor \frac a2\rfloor} = p^{a^2/4+O(1)}.
\]
\end{lemma}

\begin{proof}
Suppose first that $a=2c$ is even. Using $(2c-b)b = c^2-(c-b)^2 \le c^2-(c-b)$ for $b\le c-1$, we obtain
\[
\sum_{b=0}^a p^{(a-b)b} = p^{c^2} + 2 \sum_{b=0}^{c-1} p^{(2c-b)b} = p^{c^2} + O\bigg( \sum_{b=0}^{c-1} p^{c^2-(c-b)} \bigg) = p^{c^2} + O\big( p^{c^2-1} \big),
\]
which is certainly of the form $p^{a^2/4+O(1)}$. Even more simply, $p^{(a-\lfloor \frac a2\rfloor)\lfloor \frac a2\rfloor} = p^{(2c-c)c} = p^{a^2/4}$ exactly.

Now suppose that $a=2c+1$ is odd. Using $(2c+1-b)b = (c+\frac12)^2-(c+\frac12-b)^2 \le (c+\frac12)^2 - (c+\frac12-b)$ for $b\le c-1$, we obtain
\begin{align*}
\sum_{b=0}^a p^{(a-b)b} &= 2 \bigg( p^{c(c+1)} + \sum_{b=0}^{c-1} p^{(2c+1-b)b} \bigg) \\
&= 2 \bigg( p^{c(c+1)} + O\bigg( \sum_{b=0}^{c-1} p^{(c+\frac12)^2 - (c+\frac12-b)} \bigg) \bigg) = 2p^{c(c+1)} + O\big( p^{(c+\frac12)^2-\frac12} \big).
\end{align*}
Since $(c+\frac12)^2-\frac12 < c(c+1)$, the right-hand side is $\asymp p^{c(c+1)} = p^{(c+\frac12)^2-\frac14}$, which is also of the form $p^{a^2/4+O(1)}$. On the other hand, $p^{(a-\lfloor \frac a2\rfloor)\lfloor \frac a2\rfloor} = p^{(2c+1-c)c} = p^{a^2/4+O(1)}$ as we have just seen.
\end{proof}

\begin{prop}
\label{log asymptotic p subgroups prop}
For any prime $p$ and any partition $\balpha$,
\[
\log N_p(\balpha) = \frac{\log p}4 \sum_{j=1}^{\alpha_1} a_j^2 + O(\alpha_1\log p).
\]
\end{prop}

\begin{proof}
We recall our notation $\ba = (a_1,\dots,a_{\alpha_1},0,\dots)$ and $\bb = (b_1,\dots,b_{\alpha_1},0,\dots)$ for the conjugate partitions of $\balpha$ and $\bbeta$, respectively.
Since $N_p(\balpha) = \sum_{\bbeta \preceq \balpha} N_p(\balpha,\bbeta)$ by definition, Lemms~\ref {subgroups fixed type lemma} tells us that there exist constants $0\le\theta_\bbeta<6$ and $0\le\theta<6$ such that
\begin{equation}  \label{gotta fix order of summation}
N_p(\balpha) = \sum_{\bbeta \preceq \balpha} \bigg( \prod_{j=1}^{\alpha_1} p^{(a_j-b_j)b_j} \bigg) \big( 1+\theta_\bbeta p^{-1} \big)^{\alpha_1}= \big( 1+\theta p^{-1} \big)^{\alpha_1} \sum_{\bbeta \preceq \balpha} \bigg( \prod_{j=1}^{\alpha_1} p^{(a_j-b_j)b_j} \bigg),
\end{equation}
where the second equality again uses the intermediate value property of $f(\theta) = \big( 1+\theta p^{-1} \big)^{\alpha_1}$ (and the positivity of each summand). On one hand, since every $\bbeta \preceq \balpha$ corresponds to certain choices $0\le b_j\le a_j$, we have
\[
\sum_{\bbeta \preceq \balpha} \bigg( \prod_{j=1}^{\alpha_1} p^{(a_j-b_j)b_j} \bigg) \le \sum_{b_1=0}^{a_1} \cdots \sum_{b_{\alpha_1}=0}^{a_{\alpha_1}} \prod_{j=1}^{\alpha_1} p^{(a_j-b_j)b_j} = \prod_{j=1}^{\alpha_1} \sum_{b_j=0}^{a_j} p^{(a_j-b_j)b_j} = \prod_{j=1}^{\alpha_1} p^{a_j^2/4+O(1)}
\]
by Lemma~\ref {odd even sum lemma}. On the other hand, let $\bbeta_1$ be the subpartition of $\balpha$ whose conjugate partition is $\bb = ( \lfloor \frac{a_1}2 \rfloor, \dots, \lfloor \frac{a_{\alpha_1}}2 \rfloor, 0,\dots )$; considering only the summand on the right-hand side of equation~\eqref{gotta fix order of summation} corresponding to $\bbeta=\bbeta_1$ yields
\[
\sum_{\bbeta \preceq \balpha} \bigg( \prod_{j=1}^{\alpha_1} p^{(a_j-b_j)b_j} \bigg) \ge \prod_{j=1}^{\alpha_1} p^{(a_j-\lfloor \frac{a_j}2 \rfloor)\lfloor \frac{a_j}2 \rfloor} = \prod_{j=1}^{\alpha_1} p^{a_j^2/4+O(1)}
\]
by Lemma~\ref {odd even sum lemma} again. Combining these last two inequalities with equation~\eqref{gotta fix order of summation}, we conclude that
\[
N_p(\balpha) = \big( 1+\theta p^{-1} \big)^{\alpha_1} \prod_{j=1}^{\alpha_1} p^{a_j^2/4+O(1)}, 
\]
and therefore (since $\log(1+x)$ is bounded for $0\le x<3$)
\[
\log N_p(\balpha) = O(\alpha_1) + \sum_{j=1}^{\alpha_1} \bigg( \frac{a_j^2}4 + O(1) \bigg) \log p = \frac{\log p}4 \sum_{j=1}^{\alpha_1} a_j^2 + O(\alpha_1\log p)
\]
as claimed.
\end{proof}

\subsection{Counting $p$-subgroups of the multiplicative group}

We now begin the proof of Proposition~\ref{log G as a polynomial intro version} in earnest. As in the introduction,
let $G(n)$ denote the number of subgroups of $\Znt$ and let $G_p(n)$ denotes the number of $p$-subgroups of~$\Znt$. Since every finite abelian group is the direct product of its $p$-Sylow subgroups, it is easy to see that
\[
G(n) = \prod_{p \mid \phi(n)} G_p(n) \quad\text{and thus}\quad \log G(n) = \sum_{p \mid \phi(n)} \log G_p(n).
\]
Therefore, we first turn our attention to $\log G_p(n)$. It turns out that $\log G_p(n)$ can be expressed in terms of arithmetic functions $\ovomega_{p^j}(n)$, defined in two stages as follows:

\begin{definition} \label{omega q def}
For any positive integer $q$, let $\omega_q(n)$ denote the number of distinct primes $p \mid n$ such that $p \equiv 1 \mod q$. For example, $\omega_1(n)=\omega(n)$, while $\omega_2(n) = \omega(n)-1$ when $n$ is even and $\omega_2(n) = \omega(n)$ when $n$ is odd.
\end{definition}

These functions $\omega_q$ will play a prominent role in the remainder of this paper.
Already we start forming our intuition: since $\omega(n)$ is typically about $\log\log n$, and since one in every $\phi(q)$ primes on average is congruent to $1\mod q$, the function $\omega_q(n)$ is typically about $\frac1{\phi(q)} \log\log n$ in size; and indeed, an Erd{\H o}s--Kac law for $\omega_q(n)$ itself is straightforward to derive from the results in~\cite{ek40}.

We must make a punctilious alteration to these functions $\omega_q$ in order for them to exactly describe the structure of $\Znt$. However, our intuition should also include the understanding that the difference between $\omega_q$ and its sibling $\ovomega_q$ (defined momentarily) is negligible in the distributional sense; in particular, all we will really use is that $\ovomega_q(n) = \omega_q(n) + O(1)$ uniformly in integers $n$ and prime powers~$q$. Recall that the notation $p^r \parallel m$ means that $p^r \mid m$ but $p^{r+1} \nmid m$.

\begin{definition}
For any prime power $p^r$, define
\[
\ovomega_{p^r}(n) = \begin{cases}
\omega_{p^r}(n)+1, &\text{if $p$ is odd and } p^{r+1}\mid n, \\
\omega_{p^r}(n), &\text {if $p$ is odd and } p^{r+1}\nmid n, \\
\omega_{2}(n)+2, &\text{if $p^r=2^1$ and } 2^3\mid n, \\
\omega_{2}(n)+1, &\text{if $p^r=2^1$ and } 2^2\parallel n, \\
\omega_{2}(n), &\text{if $p^r=2^1$ and } 2^2\nmid n, \\
\omega_{2^r}(n)+1, &\text{if $p=2$ and $r>1$ and } p^{r+2}\mid n, \\
\omega_{2^r}(n), &\text{if $p=2$ and $r>1$ and } p^{r+2}\nmid n.
\end{cases}
\]
\end{definition}

\begin{definition}  \label{lambda p def}
For any prime $p$ and any positive integer $n$, let $\lambda_p(n)$ denote the largest power of $p$ that divides the Carmichael function $\lambda(n)$. In other words, $\lambda_p(n)$ is the exponent of the $p$-Sylow subgroup of $\Znt$.
\end{definition}

\begin{lemma}  \label{p Sylow structure lemma}
Let $n$ be a positive integer, and let $p$ be a prime dividing $\phi(n)$. The $p$-Sylow subgroup of $\Znt$ is isomorphic to
$\Z_{p^{\alpha_1}} \times \Z_{p^{\alpha_2}} \times \cdots$,
where $\balpha = (\alpha_1,\alpha_2,\dots)$ is the conjugate partition to
\[
\ba = (a_1,a_2,\dots) = (\ovomega_p(n), \ovomega_{p^2}(n), \dots, \ovomega_{p^{\lambda_p(n)}}(n), 0, \dots).
\]
\end{lemma}

\begin{proof}
First let $p$ be an odd prime. Write the $p$-Sylow subgroup of $\Znt$ as $\Z_{p^{\alpha_1}} \times \Z_{p^{\alpha_2}} \times \cdots$ for some partition $\balpha$ which we want to determine. There are two possible sources of factors of $p$ in $\phi(n)$: primes $q\mid n$ such that $q \equiv 1 \mod p$ (including those congruent to $1$ modulo higher powers of $p$), and $p^2$ itself (or a higher power of $p$) dividing $n$. Furthermore, by the Chinese remainder theorem and the existence of primitive roots modulo every odd prime power, we can say exactly how each of these sources affects the $p$-Sylow subgroup of~$\Znt$.

Each prime $q \mid n$ such that $q \equiv 1 \mod {p^j}$ contributes, to the the $p$-Sylow subgroup of $\Znt$, a factor of $\Z_{p^m}$ with $m \geq j$ (indeed, $m$ is the exponent of $p$ in the prime factorization of $q-1$). Moreover, if $p^{j+1} \mid n$, then this power of $p$ contributes to the $p$-Sylow subgroup of $\Znt$ another factor of $\Z_{p^m}$ with $m\ge j$ (in this case, $m+1$ is the exponent of $p$ in the prime factorization of~$n$ itself). All factors of the form $\Z_{p^m}$ in the primary decomposition of $\Znt$ arise in one of these two ways; therefore, the number of factors of order at least $p^j$ in the $p$-Sylow subgroup of $\Znt$ is exactly equal to $\ovomega_{p^j}(n)$. But $a_j$, the $j$th entry in the conjugate partition to $\balpha$, is precisely the number of factors of order at least $p^j$ in $\Z_{p^{\alpha_1}} \times \Z_{p^{\alpha_2}} \times \cdots$. We conclude that $a_j = \ovomega_{p^j}(n)$ as desired.

The case $p = 2$ follows by an similar analysis, complicated slightly by the fact that $\Z_2^\times \cong \Z_1$ and $\Z_4^\times \cong \Z_2$ while $\Z_{2^r}^\times \cong \Z_{2^{r-2}} \times \Z_2$ when $r\ge3$.
\end{proof}

It is worth remarking that in particular, Lemma~\ref{p Sylow structure lemma} shows that the exponent of $p$ in the prime factorization of $\phi(n)$ is exactly $\sum_{j=1}^{\lambda_p(n)} \ovomega_{p^j}(n)$ for every prime~$p$. Consequently,
\begin{equation} \label{sum of its parts}
\sum_{p\mid\phi(n)} \sum_{j=1}^{\lambda_p(n)} \ovomega_{p^j}(n) \log p = \log\phi(n).
\end{equation}
Furthermore, let $\nu_p(n)$ denote the power of $p$ in the prime factorization of~$n$. The proof of Lemma~\ref{p Sylow structure lemma} also shows that for odd primes $p$,
\[
\lambda_p(n) = \max \big\{ \nu_p(n)-1, \max\{ j\colon \omega_{p^j}(n) \ge 1 \} \big\};
\]
when $p=2$, we must replace $\nu_p(n)-1$ with $\max\{0,\nu_2(n)-2\}$. In either case,
\begin{equation} \label{lambdapn bound}
\lambda_p(n) \le \max \bigg\{ \nu_p(n), \sum_{j\ge1} \omega_{p^j}(n) \bigg\}.
\end{equation}

With the following proposition, we may leave most of the details of abelian groups and partitions behind and operate within the realm of analytic number theory to complete the proof of Proposition~\ref{log G as a polynomial intro version}.

\begin{prop}  \label{gpn as a polynomial}
\label{translating to Wn lemma}
For any positive integer $n$ and any prime $p$ dividing $\phi(n)$,
\[
\log G_p(n) = \frac{\log p}4 \sum_{j=1}^{\lambda_p(n)} \ovomega_{p^j}(n)^2 + O(\lambda_p(n) \log p).
\]
Moreover, if $p \parallel \phi(n)$, then $\log G_p(n) = \log 2$.
\end{prop}

\begin{proof}
If $p \parallel \phi(n)$, then the $p$-part of $\Znt$ is precisely $\Z_p$, which trivially contains exactly two subgroups; hence $G_p(n) = 2$ in this case. In general, Proposition~\ref{log asymptotic p subgroups prop} tells us that
\[
\log N_p(\balpha) = \frac{\log p}4 \sum_{j=1}^{\alpha_1} a_j^2 + O(\alpha_1\log p),
\]
while Lemma~\ref{p Sylow structure lemma} gives us the exact evaluations $\alpha_1=\lambda_p(n)$ and $a_j = \ovomega_{p^j}(n)$.
\end{proof}

\subsection{Counting all subgroups of the multiplicative group}

The main goal of this section is to establish Proposition~\ref{log G as a polynomial intro version}, which says that $\log G(n)$ is approximately equal to a particular polynomial expression in additive functions of $n$, at least for most integers~$n$.

Several times in the course of these proofs, we will make use of upper bounds (of the correct order of magnitude) that follow, via partial summation, from the prime number theorem, or indeed from Mertens's formulas or even Chebyshev's bounds for prime-counting functions. Such sums include sums over primes like $\sum_{p\le y} 1/p$ or $\sum_{p>y} 1/p^2$, or sums over prime powers like $\sum_{p^j \le y} \log^2(p^j)/p^j$ or $\sum_{q\le y} \Lambda(q)/q$. Moreover, since $q/\phi(q) \le 2$ for all prime powers $q$, such sums can also be modified to have denominators of $p-1$ instead of~$p$, or $\phi(q)$ instead of~$q$. In all such cases, we shall simply say ``by partial summation'' to indicate that the required upper bounds follows in a standard way from these prime-counting estimates.

In addition, we will make frequent use of the following Mertens-type estimate for arithmetic progressions, which can be found in~\cite{nor} or~\cite{pom}:
\begin{lemma}\label{mertenspartial}
For $2 \leq q \leq x$, we have $\displaystyle \sum_{\substack{p \leq x \\ p \equiv 1 \mod q}} \frac{1}{p} = \frac{\log\log x}{\phi(q)} + O\bigg( \frac{\log q}{\phi(q)} \bigg)$.
\end{lemma}

For the rest of this section, we set
\begin{align*}
W &= \log\log\log x \\
X &= (\log\log x)^{1/2}(\log\log\log x)^2 \\
Y &= (\log\log x)^2.
\end{align*}
(Note that this definition of $X$ is the same as in Proposition~\ref{log G as a polynomial intro version}.)

\begin{lemma} \label{triple threat lemma}
For all but $O(x/W)$ integers $n\le x$,
\begin{multline}  \label{four expressions}
\max\bigg\{ \sum_{\substack{p\le Y \\ \lambda_p(n)\ge1}} \log p,\, \sum_{p\le Y} \lambda_p(n)\log p ,\, \sum_{p\le Y} \nu_p(n)\log p ,\, \sum_{p^j\le Y} \omega_{p^j}(n)\log p \bigg\} \\
\ll \log\log x \cdot (\log\log\log x)^2.
\end{multline}
\end{lemma}

\begin{proof}
The first sum on the left-hand side of equation~\eqref{four expressions} is clearly bounded above by the second sum; and this second sum, by equation~\eqref{lambdapn bound}, is bounded above by the maximum of third and fourth sums on the left-hand side. It therefore suffices to show that
\begin{equation} \label{summed both over n}
\sum_{n\le x} \sum_{p\le Y} \nu_p(n)\log p + \sum_{n\le x} \sum_{p^j\le Y} \omega_{p^j}(n)\log p \ll x\log\log x \cdot \log Y,
\end{equation}
for then there can be no more than $O(x/W)$ integers $n\le x$ for which either of the two summands exceeds $\log\log x\cdot\log Y\cdot W = \log\log x \cdot (\log\log\log x)^2$.

The first sum on the left-hand side of equation~\eqref{summed both over n} can be bounded simply:
\begin{align*}
\sum_{n\le x} \sum_{p\le Y} \nu_p(n)\log p = \sum_{n\le x} \sum_{p\le Y} \log p \sum_{\substack{j\ge1 \\ p^j\mid n}} 1 &= \sum_{p\le Y} \log p \sum_{j\ge1} \sum_{\substack{n\le x \\ p^j\mid n}} 1 \\
&\le \sum_{p\le Y} \log p \sum_{j\ge1} \frac x{p^j} = x \sum_{p\le Y} \frac{\log p}{p-1} \ll x\log Y
\end{align*}
by partial summation, which is more than sufficient. As for the second sum on the left-hand side of equation~\eqref{summed both over n},
\begin{align*}
\sum_{n\le x} \sum_{p^j\le Y} \omega_{p^j}(n)\log p = \sum_{n\le x} \sum_{p^j\le Y} \log p \sum_{\substack{q\mid n \\ q\equiv1\mod{p^j}}} 1 &= \sum_{p^j\le Y} \log p \sum_{\substack{q\le x \\ q\equiv1\mod{p^j}}} \sum_{\substack{n\le x \\ q\mid n}} 1 \\
&\le \sum_{p^j\le Y} \log p \sum_{\substack{q\le x \\ q\equiv1\mod{p^j}}} \frac xq.
\end{align*}
Since $Y < \log x$ when $x$ is large enough, Lemma~\ref{mertenspartial} yields
\begin{align*}
\sum_{n\le x} \sum_{p^j\le Y} \omega_{p^j}(n)\log p &\le x \sum_{p^j\le Y} \log p \bigg( \frac{\log\log x}{\phi(p^j)} + O\bigg( \frac{\log(p^j)}{\phi(p^j)} \bigg) \bigg) \\
&\ll x\log\log x \sum_{p^j\le Y} \frac{\log p}{\phi(p^j)} + x \sum_{p^j\le Y} \frac{\log^2(p^j)}{\phi(p^j)} \\
&\ll x\log\log x \cdot \log Y + x\log^2 Y
\end{align*}
by partial summation, completing the verification of the bound~\eqref{summed both over n}.
\end{proof}

The following lemma is very similar to known results (see~\cite{ep85} for example) on the scarcity of numbers $n$ for which $\phi(n)$ is divisible by the square of a large prime.

\begin{lemma}  \label{truncate at A lemma}
All but $O(x/W)$ integers $n\le x$ have both $\lambda_p(n)\le1$ for all $p>Y$ and $\omega_{p^j}(n)\le1$ for all $p^j>Y$. 
\end{lemma}

\begin{proof}
First, fix a prime $p>Y$. If $\lambda_p(n)\ge2$, then either $p^3\mid n$ or there exists a prime $q\mid n$ with $q\equiv1\mod{p^2}$; the number of integers $n\le x$ satisfying one of these two conditions is at most
\[
\frac x{p^3} + \sum_{\substack{q\le x \\ q\equiv1\mod{p^2}}} \frac xq.
\]
Therefore the total number of integers $n\le x$ for which $\lambda_p(n)\ge2$ for even a single prime $p>Y$ is, by Lemma~\ref{mertenspartial}, at most
\begin{align*}
\sum_{Y<p\le x^{1/3}} \frac x{p^3} &{}+ \sum_{Y<p\le\sqrt x} \sum_{\substack{q\le x \\ q\equiv1\mod{p^2}}} \frac xq \\
&< \sum_{p>Y} \frac x{p^3} + x \sum_{p>Y} \bigg( \frac{\log\log x}{\phi(p^2)} + O\bigg( \frac{\log p^2}{\phi(p^2)} \bigg) \bigg) \\
&\ll \frac x{Y^2\log Y} + x \bigg( \frac{\log\log x}{Y\log Y} + \frac1Y \bigg)
\end{align*}
by partial summation; this is an acceptably small bound for the number of such $n\le x$, given our choices of $Y$ and~$W$.

Similarly, fix a prime power $p^j>Y$. If $\omega_{p^j}(n) \ge 2$, then there exist two distinct primes $q$ and $r$ dividing $n$ such that $q\equiv r\equiv1\mod{p^j}$. The number of integers $n\le x$ satisfying this condition is, by Lemma~\ref{mertenspartial}, at most
\begin{align*}
\sum_{\substack{q<r\le x \\ q\equiv r\equiv1\mod{p^j}}} \frac x{qr} &< \frac x2 \bigg( \sum_{\substack{q\le x \\ q\equiv1\mod{p^j}}} \frac 1q \bigg)^2 \\
&\ll x \bigg( \frac{\log\log x}{\phi(p^j)} + \frac{\log p^j}{\phi(p^j)} \bigg)^2 \ll x \bigg( \frac{(\log\log x)^2}{(p^j)^2} + \frac{(\log p^j)^2}{(p^j)^2} \bigg).
\end{align*}
Therefore the total number of integers $n\le x$ for which $\omega_{p^j}(n) \ge 2$ for even a single prime power $p^j>Y$ is
\begin{align*}
\ll \sum_{p^j>Y} x \bigg( \frac{(\log\log x)^2}{(p^j)^2} + \frac{(\log p^j)^2}{(p^j)^2} \bigg) \ll x \bigg( \frac{(\log\log x)^2}{Y\log Y} + \frac{\log Y}Y \bigg) \ll \frac xW
\end{align*}
again by partial summation.
\end{proof}

We now have collected enough results to obtain a not-quite-final version of Proposition~\ref{log G as a polynomial intro version} where, for the moment, the range of summation ($p^j\le Y$ rather than $p^j\le X$) is longer than we would like.

\begin{lemma}  \label{Y instead of X lemma}
For all but $O(x/W)$ integers $n\le x$,
\begin{equation*}
\log G(n) = \omega(\phi(n)) \log 2 + \tfrac14 \sum_{p^j\le Y} \omega_{p^j}(n)^2 \log p + O(\log\log x \cdot (\log\log\log x)^2).
\end{equation*}
\end{lemma}

\begin{proof}
By Proposition~\ref{gpn as a polynomial},
\begin{align*}
\log G(n) &= \sum_{p\mid\phi(n)} \log G_p(n) \\
&= \sum_{\substack{p\mid\phi(n) \\ \lambda_p(n)\omega_p(n)=1}} \log 2 + \sum_{\substack{p\mid\phi(n) \\ \lambda_p(n)\omega_p(n)\ge2}} \bigg( O(\lambda_p(n) \log p) + \frac{\log p}4 \sum_{j=1}^{\lambda_p(n)} \ovomega_{p^j}(n)^2 \bigg) \\
&= \sum_{\substack{p\mid\phi(n) \\ \lambda_p(n)\omega_p(n)=1}} \log 2 + O\bigg( \sum_{\substack{p\mid\phi(n) \\ \lambda_p(n)\omega_p(n)\ge2}}\lambda_p(n) \log p \bigg) + \sum_{\substack{p\mid\phi(n) \\ \lambda_p(n)\omega_p(n)\ge2 \\ j\le\lambda_p(n)}} \tfrac14 \ovomega_{p^j}(n)^2 \Lambda(p^j) \\
&= \sum_{p\mid\phi(n)} \log 2 + O\bigg( \sum_{\substack{p\mid\phi(n) \\ \lambda_p(n)\omega_p(n)\ge2}} (\log2 + \lambda_p(n) \log p) \bigg) + \sum_{\substack{p\mid\phi(n) \\ \lambda_p(n)\omega_p(n)\ge2 \\ j\le\lambda_p(n)}} \tfrac14 \ovomega_{p^j}(n)^2 \Lambda(p^j).
\end{align*}
Since $\lambda_p(n) \ge1$ for all $p\mid\phi(n)$,
\[
\sum_{\substack{p\mid\phi(n) \\ \lambda_p(n)\omega_p(n)\ge2}} (\log2 + \lambda_p(n) \log p) \ll \sum_{\substack{p\mid\phi(n) \\ \lambda_p(n)\omega_p(n)\ge2}} \lambda_p(n) \log p.
\]
Furthermore, by Lemma~\ref{truncate at A lemma}, for all but $O(x/W)$ integers $n\le x$ we never have $\lambda_p(n)\ge2$ for any $p>Y$; for these non-exceptional integers, we can therefore incorporate the condition $p\le Y$ into the relevant sums, yielding
\begin{align*}
\log G(n) = \sum_{p\mid\phi(n)} \log 2 + O\bigg( \sum_{p\le Y} \lambda_p(n) \log p \bigg) + \sum_{\substack{p\le Y \\ p\mid\phi(n) \\ \lambda_p(n)\omega_p(n)\ge2 \\ j\le\lambda_p(n)}} \tfrac14 \ovomega_{p^j}(n)^2 \Lambda(p^j).
\end{align*}
In this last sum, for all but $O(x/W)$ integers $n\le x$, Lemma~\ref{truncate at A lemma} also implies that $\omega_{p^j}(n) \le 1$ (and thus $\ovomega_{p^j}(n) \ll 1$) for all $p^j > Y$, which implies
\[
\sum_{\substack{p\le Y \\ p\mid\phi(n) \\ \lambda_p(n)\omega_p(n)\ge2 \\ j\le \lambda_p(n) \\ p^j > Y}} \ovomega_{p^j}(n)^2 \Lambda(p^j) \ll \sum_{\substack{p\le Y \\ p\mid\phi(n) \\ \lambda_p(n)\omega_p(n)\ge2 \\ j\le \lambda_p(n) \\ p^j > Y}} 1\cdot \Lambda(p^j) \le \sum_{\substack{p\le Y \\ j\le \lambda_p(n)}} 1\cdot \Lambda(p^j) = \sum_{p\le Y} \lambda_p(n)\log p;
\]
therefore for these non-exceptional integers, we can strengthen the condition $p\le Y$ to $p^j\le Y$ in the last sum to obtain
\begin{align}
\log G(n) &= \sum_{p\mid\phi(n)} \log 2 + O\bigg( \sum_{p\le Y} \lambda_p(n) \log p \bigg) + \sum_{\substack{p\mid\phi(n) \\ \lambda_p(n)\omega_p(n)\ge2 \\ j\le \lambda_p(n) \\ p^j\le Y}} \tfrac14 \ovomega_{p^j}(n)^2 \Lambda(p^j) \notag \\
&= \omega(\phi(n)) \log 2 + \sum_{\substack{p\mid\phi(n) \\ \lambda_p(n)\omega_p(n)\ge2 \\ j\le \lambda_p(n) \\ p^j\le Y}} \tfrac14 \ovomega_{p^j}(n)^2 \Lambda(p^j) + O\big( (\log\log x)^2\log\log\log x \big),  \label{bars present}
\end{align}
where the second equality is valid for all but $O(x/W)$ integers $n\le x$ by Lemma~\ref{triple threat lemma}. From the definition of $\ovomega_{p^j}$, we know that $\ovomega_{p^j}(n) = \omega_{p^j}(n) + O(1)$, and consequently $\ovomega_{p^j}(n)^2 = \omega_{p^j}(n)^2 + O\big( \omega_{p^j}(n)+1 \big)$. In particular,
\begin{align*}
\sum_{\substack{p\mid\phi(n) \\ \lambda_p(n)\omega_p(n)\ge2 \\ j\le \lambda_p(n) \\ p^j\le Y}} \tfrac14 \ovomega_{p^j}(n)^2 \log p &= \sum_{\substack{p\mid\phi(n) \\ \lambda_p(n)\omega_p(n)\ge2 \\ j\le \lambda_p(n) \\ p^j\le Y}} \big( \tfrac14 \omega_{p^j}(n)^2 + O( \omega_{p^j}(n)+1 ) \big) \log p \\
&= \sum_{\substack{p\mid\phi(n) \\ \lambda_p(n)\omega_p(n)\ge2 \\ j\le \lambda_p(n) \\ p^j \le Y}} \tfrac14 \omega_{p^j}(n)^2 \log p + O\bigg( \sum_{p^j\le Y} \omega_{p^j}(n) \log p + \sum_{p\le Y} \lambda_p(n) \log p \bigg),
\end{align*}
and by Lemma~\ref{triple threat lemma} this error term is also $\ll \log\log x \cdot (\log\log\log x)^2$ for all but $O(x/W)$ integers $n\le x$. Therefore we may modify equation~\eqref{bars present} to
\begin{equation*}
\log G(n) = \omega(\phi(n)) \log 2 + \sum_{\substack{p\mid\phi(n) \\ \lambda_p(n)\omega_p(n)\ge2 \\ j\le \lambda_p(n) \\ p^j\le Y}} \tfrac14 \omega_{p^j}(n)^2 \log p + O(\log\log x \cdot (\log\log\log x)^2).
\end{equation*}
In this sum, we may remove the condition of summation $\lambda_p(n)\omega_p(n)\ge2$ at a cost of at most $\sum_{p\le Y} \frac14\lambda_p(n)\log p$, which again is negligible for all but $O(x/W)$ integers $n\le x$ by Lemma~\ref{triple threat lemma}. Since $\omega_{p^j}(n)=0$ whenever $p\nmid\phi(n)$ or $j>\lambda_p(n)$, we may remove the conditions $p\mid\phi(n)$ and $j\le\lambda_p(n)$ as well. This establishes the lemma.
\end{proof}

Finally, we show that we can truncate the range of summation in the above lemma from $p^j \le Y$ down to $p^j \leq X$ at the cost of a larger error term, thereby obtaining Proposition~\ref{log G as a polynomial intro version}.

\begin{proof}[Proof of Proposition~\ref{log G as a polynomial intro version}]
In the notation of Proposition~\ref{log G as a polynomial intro version} and of this section, Lemma~\ref{Y instead of X lemma} states that for all but $x/W$ integers $n\le x$,
\begin{equation*}
\log G(n) = P_n(x) + \tfrac14 \sum_{X < q \le Y} \omega_q(n)^2 \Lambda(q) + O(\log\log x\cdot(\log\log\log x)^2).
\end{equation*}
Therefore it suffices to show that for all but $x/W$ integers $n\le x$, the sum in the equation above is $\ll (\log\log x)^{3/2}/\log\log\log x$. In turn, this statement can be established by showing that
\begin{equation}  \label{log G as a polynomial estimate}
\sum_{n\le x} \sum_{X < q \le Y} \omega_q(n)^2 \Lambda(q) \ll \frac{x(\log\log x)^{3/2}}{(\log\log\log x)^2}.
\end{equation}

We may write
\begin{align*}
\sum_{n \leq x} \sum_{X < q \leq Y} \omega_q(n)^2 \Lambda(q) &= \sum_{n \leq x} \sum_{X < q \leq Y} \Lambda(q) \bigg( \sum_{\substack{p \mid n \\ p \equiv 1 \mod q}} 1 \bigg)^2 \\
&= \sum_{n \leq x} \sum_{X < q \leq Y} \Lambda(q) \sum_{\substack{p_1,p_2 \mid n \\ p_1 \equiv p_2 \equiv 1 \mod q}} 1 \\
&= \sum_{X < q \leq Y} \Lambda(q) \sum_{\substack{p_1 \equiv p_2 \equiv 1 \mod q}} \sum_{\substack{n \leq x \\ p_1,p_2 \mid n}} 1 \\
&= \sum_{X < q \leq Y} \Lambda(q) \sum_{\substack{p \equiv 1 \mod q}} \sum_{\substack{n \leq x \\ p \mid n}} 1 +  \sum_{X < q \leq Y} \Lambda(q) \sum_{\substack{p_1 \equiv p_2 \equiv 1 \mod q \\ p_1\ne p_2}} \sum_{\substack{n \leq x \\ p_1p_2 \mid n}} 1.
\end{align*}
For the first sum, Lemma~\ref{mertenspartial} gives
\begin{align*}
\sum_{n \leq x} \sum_{X < q \leq Y} \Lambda(q) \sum_{\substack{p \mid n \\ p \equiv 1 \mod q}} 1 &= \sum_{X < q \leq Y} \Lambda(q) \sum_{\substack{p \leq x \\ p \equiv 1 \mod q}} \sum_{\substack{n \leq x \\ p \mid n}} 1 \\
&\ll x \sum_{X < q \leq Y} \Lambda(q) \sum_{\substack{p \leq x \\ p \equiv 1 \mod q}} \frac{1}{p} \\
&\ll x\log\log x \sum_{X < q \leq Y} \frac{\Lambda(q)}{\phi(q)} \ll x\log\log x\cdot \log Y
\end{align*}
by partial summation. For the second sum, we argue similarly:
\begin{align*}
\sum_{n \leq x} \sum_{X < q \leq Y} \Lambda(q) \sum_{\substack{p_1, p_2 \leq x \\ p_1p_2 \mid n \\ p_1 \equiv p_2 \equiv 1 \mod q}} 1 &\ll \sum_{X < q \leq Y} \Lambda(q) \sum_{\substack{p_1, p_2 \\ p_1 \equiv p_2 \equiv 1 \mod q}} \frac{x}{p_1p_2} \\
&\ll x(\log\log x)^2 \sum_{q > X} \frac{\Lambda(q)}{q^2} \ll x(\log\log x)^{2}/X = \frac{x(\log\log x)^{3/2}}{(\log\log\log x)^2}.
\end{align*}
These last two upper bounds establish the estimate~\eqref{log G as a polynomial estimate} and therefore the proposition.
\end{proof}

\section{Notation and setup}\label{notation and setup}

In this section, we prepare some notation we will need to prove Proposition~\ref{moment asymptotics}. At the end of the section, we outline the main stages of the proof, which span the next several sections.

\begin{definition}  \label{omega0 def}
Define the function
\[
\omega_0(n) = \omega(\phi(n)).
\]
Comparing with Definition~\ref{omega q def} shows that this notation $\omega_0$ is mathematically dubious, but it will be typographically convenient. For example, we note that $\omega_q(p) \ll \log z$ for every $q\ge0$ and every prime $p\le z$: when $q\ge2$ this is obvious from Definition~\ref{omega q def}, while for $q=0$ we have $\omega_0(p) = \omega(p-1) \le \log(p-1)/\log 2$.

In this notation, the definitions~\eqref{Qnx definition} and~\eqref{Dx definition} become
\begin{equation} \label{Pn D new notation}
\begin{split}
P_n(x) &= \log 2 \cdot \omega_0(n) + \frac14 \sum_{2\le q\le X} \omega_{q}(n)^2 \Lambda(q) \\
D(x) &= \log 2 \cdot \mu(\omega_0) + \frac{1}{4}\sum_{2\le q \leq X} \mu(\omega_q)^2 \Lambda(q),
\end{split}
\end{equation}
where (by equation~\eqref{mu definition}) we may simply write
\[
\mu(\omega_q) = \sum_{p \leq x} \frac{\omega_q(p)}{p}
\]
for every $q\ge0$. We shall continue to write ranges of summation over $q$ as either $2\le q$, when the sum excludes $q=0$, or $0\le q$, when the sum includes~$q=0$.
\end{definition}

By way of intuition, the typical size of $\omega_0(n)$ is $\frac12(\log\log x)^2$; this is quite a bit larger than the typical size of any $\omega_q(n)$ with $q\ge2$ but, on the other hand, these $\omega_q$ typically occur squared, while the function $\omega_0$ typically occurs to the first power. Consequently, the contribution of the two types of function to the typical size of $P_n(x)$ is of the same order of magnitude. The typical size of $P_n(x)$, as $n$ varies over integers up to $x$, is asymptotically $D(x)$ (due essentially to ``linearity of expectation''), and consequently the distribution of the difference $P_n(x)-D(x)$ will be the main focus of our investigation.

\begin{definition}  \label{fF def}
For any prime $p$, define the function
\[
f_p(a) = \begin{cases} 1-1/p, &\text{if } p\mid a, \\ -1/p, &\text{if } p \nmid a.\end{cases}
\]
We extend this function completely multiplicatively in the subscript (not, as might be expected, in the argument): for any positive integer $r$, we set
\[
f_r(a) = \prod_{p^\alpha \parallel r} f_p(a)^\alpha.
\]
Finally, we set
\[
F_{\omega_q}(a) = \sum_{p\leq x} \omega_q(p) f_p(a).
\]
\end{definition}

Notice that, for any $n \leq x$, we have the exact identity
\begin{equation}\label{function equals mean plus F}
\omega_q(n) = \mu(\omega_q) + F_{\omega_q}(n).
\end{equation}
We have thereby decomposed an additive function into its mean value on the integers up to $x$ (which is asymptotically equal to $\mu(\omega_q)$) and a term $F_{\omega_q}(n)$ that oscillates as $n$ varies. This innovation, due to Granville and Soundararajan~\cite{gs07}, allows for a more direct identification of the main terms that arise in the calculations of the $h$th moments
\[
M_h(x) = \sum_{n \leq x} (P_n(x) - D(x))^h
\]
of the difference between $P_n(x)$ and its mean value. We approach these moments
by first rewriting $P_n(x)$ using equation~(\ref{function equals mean plus F}):
\begin{align*}
P_n(x) &= \log 2 \cdot \omega_0(n) + \frac14 \sum_{2\le q\le X} \omega_{q}(n)^2 \Lambda(q) \\
 &= \log 2 \cdot (\mu(\omega_0) + F_{\omega_0}(n)) + \frac14 \sum_{2\le q\le X} (\mu(\omega_q) + F_{\omega_q}(n))^2 \Lambda(q).
\end{align*}
Upon expanding the square inside the sum and then subtracting $D(x)$, we obtain
\begin{equation}\label{hth moment}
P_n(x) - D(x) = \log 2 \cdot F_{\omega_0}(n) + \frac12 \sum_{2\le q \leq X} \mu(\omega_q)F_{\omega_q}(n) + \frac14 \sum_{2\le q \leq X} F_{\omega_q}(n)^2.
\end{equation}
We then estimate the $h$th moment by taking the entire right-hand side to the $h$th power, expanding into a sum of $3^h$ terms, and estimating each term separately. The terms with the fewest $F$-factors will comprise the main term for $M_h(x)$, while the others contribute only to the error term. The bookkeeping and notation involved with tracking all of these terms is quite messy, and we have organized the remainder of the paper as follows to minimize the trauma to the reader.

The aim of Section~\ref{polynomial arithmetic} is to introduce a general algebraic framework for handling the terms that arise upon expanding the $h$th power of the right-hand side of equation~\eqref{hth moment}. Section~\ref{term estimation} is devoted to proving asymptotic estimates and formulas for those individual terms as $x$ tends to infinity. In Section~\ref{proof of moment asymptotics}, we complete the proof of Proposition~\ref{moment asymptotics} using the results of previous sections. Finally, in Section~\ref{flourish}, we quickly justify our use of probabilistic language in the statement of Theorem~\ref{main theorem} by deducing Theorem 1.1 from Theorem 1.3.

\section{Polynomial accounting}\label{polynomial arithmetic}

The main goal of this section is to establish Proposition~\ref{Rh magic Phi prop}, which is used to identify and simplify the main term of the moments $M_h(x)$ (for $h$ even) at the end of Section~\ref{proof of moment asymptotics}. The proof begins with some combinatorial arguments, concerning polynomials in many variables, which are elementary but extremely notation-intensive. Along the way, we also introduce some polynomial-related notation (Definition~\ref{Rh def}) for future use.

\begin{definition} \label{Tk def}
For any positive integer $k$, define $\Sigma_k$ to be the set of all permutations of $\{1,\dots,k\}$ (that is, the set of all bijections from $\{1,\dots,k\}$ to itself). A typical element of $\Sigma_k$ will be denoted by~$\sigma$.

For any positive even integer $k$, define $T_k$ to be the set of all 2-to-1 functions from $\{1,\dots,k\}$ to $\{1,\dots,k/2\}$. A typical element of $T_k$ will be denoted by $\tau$. We let $\tau_0$ denote the order-preserving element of $T_k$ defined by $\tau_0(j) = \lceil \frac j2 \rceil$ for each $1\le j\le k$.

For $\tau\in T_k$ and $j\in\{1,\dots,k/2\}$, define $\Upsilon_1(j)$ and $\Upsilon_2(j)$ to be the two distinct preimages of $j$ in $\{1,\dots,k\}$; we will never need to distinguish between the two.
\end{definition}

\begin{lemma} \label{Sigma to T lemma}
Let $k$ be a positive even integer. The function $\psi\colon \Sigma_k\to T_k$ defined by $\psi(\sigma) = \tau_0 \circ \sigma^{-1}$ is surjective and $2^{k/2}$-to-1.
\end{lemma}

\begin{proof}
Given any $\tau\in T_k$, the equality $\psi(\sigma) = \tau$ holds for a particular $\sigma\in\Sigma_k$ if and only if
\begin{equation} \label{Upsilons}
\{\Upsilon_1(j),\Upsilon_2(j)\} = \{\sigma(2j-1),\sigma(2j)\} \text{ for every }1\le j\le \tfrac k2.
\end{equation}
This specifies each of the $\frac k2$ unordered pairs $\{\sigma(2j-1),\sigma(2j)\}$, each of which provides a choice of two options for which element equals $\Upsilon_1(j)$ and which equals $\Upsilon_2(j)$; the total number of preimages $\sigma$ is thus exactly~$2^{k/2}$.
\end{proof}

The following definition and lemma provide one of our main tools for dealing with arbitrary powers of finite sums.

\begin{definition} \label{Phih def}
Let $h$ and $\ell$ be positive integers with $h$ even. Let $R$ be a commutative ring of characteristic zero with a unit element, and define two commutative polynomial rings over $R$ with $\ell+1$ and $(\ell+1)^2$ variables: let $x_0,\dots,x_\ell$ be indeterminates and define $S=R[x_0,\dots,x_\ell]$, and let $\{ z_{ij} \colon 0\le i, j\le \ell \}$ be indeterminates and define $\tilde S = R[z_{00},\dots,z_{\ell\ell}]$. Let $S_h$ be the $R$-submodule of $S$ spanned by monomials of total degree $h$, and let $\tilde S_{h/2}$ be the $R$-submodule of $\tilde S$ spanned by monomials of total degree~$h/2$.

We define an $R$-module homomorphism $\Phi_h \colon S_h \to \tilde S_{h/2}$ in the following way. Given a monic monomial $M=x_{m_1} \cdots x_{m_h}$ in $S_h$ (where $m_1,\dots,m_h \in \{0,\dots,\ell\}$ are not necessarily distinct), set
\[
\Phi_h(M) = \frac1{h!} \sum_{\sigma\in\Sigma_h} z_{m_{\sigma1}m_{\sigma2}} \cdots z_{m_{\sigma(h-1)}m_{\sigma h}}.
\]
(Note that the order of the indices $m_1,\dots,m_h$ is not uniquely defined by $M$, but this is not problematic since the sum defining $\Phi_h(M)$ averages over all permutations $\sigma$.)
Then we extend $\Phi_h$ $R$-linearly to $S_h$, so that $\Phi_h(\sum_j r_j M_j) = \sum_j r_j \Phi_h(M_j)$ for any monic monomials $M_j \in S_h$ and elements $r_j\in R$.

For example, with $h=4$ and $\ell=2$,
\begin{multline*}
\Phi_4(x_0^2x_1x_2 - 7x_1^3x_2) = \tfrac16z_{12}z_{00} + \tfrac16z_{21}z_{00} + \tfrac16z_{10}z_{20} + \tfrac16z_{10}z_{02} + \tfrac16z_{20}z_{01} + \tfrac16z_{01}z_{02} \\
- \tfrac72z_{11}z_{12} - \tfrac72z_{11}z_{21}.
\end{multline*}
\end{definition}

\begin{lemma} \label{monomial function lemma}
Let $h$ and $\ell$ be positive integers with $h$ even. Let $R$ be a commutative ring of characteristic zero with a unit element, and let $\Phi_h$ be defined as in Definition~\ref{Phih def}. For any elements $r_0,\dots,r_\ell\in R$,
\[
\Phi_h\big( (r_0 x_0+\cdots+r_\ell x_\ell)^h \big) = \bigg( \sum_{0\le i, j\le \ell} r_i r_j z_{ij} \bigg)^{h/2}.
\]
\end{lemma}

\begin{proof}
The key to the calculation is to purposefully avoid expanding $(r_0 x_0+\cdots+r_\ell x_\ell)^h$ using multinomial coefficients; allowing repetition such as $(x_1+x_2)^2 = x_1^2 + x_1x_2 + x_2x_1 + x_2^2$ makes the counting argument much easier. By the definition of $\Phi_h$,
\begin{align*}
\Phi_h\bigg( \bigg( \sum_{i=0}^\ell r_i x_i \bigg)^h \bigg) &= \Phi_h\bigg( \sum_{m_1=0}^\ell \cdots \sum_{m_h=0}^\ell r_{m_1} \cdots r_{m_h} x_{m_1} \cdots x_{m_h} \bigg) \\
&= \sum_{m_1=0}^\ell \cdots \sum_{m_h=0}^\ell r_{m_1} \cdots r_{m_h} \frac1{h!} \sum_{\sigma\in\Sigma_h} z_{m_{\sigma1}m_{\sigma2}} \cdots z_{m_{\sigma(h-1)}m_{\sigma h}} \\
&= \frac1{h!} \sum_{\sigma\in\Sigma_h} \sum_{m_1=0}^\ell \cdots \sum_{m_h=0}^\ell r_{m_1} \cdots r_{m_h} z_{m_{\sigma1}m_{\sigma2}} \cdots z_{m_{\sigma(h-1)}m_{\sigma h}}.
\end{align*}
Since $r_{m_1} \cdots r_{m_h} = r_{m_{\sigma1}} \cdots r_{m_{\sigma h}}$ for any $\sigma\in\Sigma_h$, we can rewrite this identity as
\[
\Phi_h\bigg( \bigg( \sum_{i=0}^\ell r_i x_i \bigg)^h \bigg) = \frac1{h!} \sum_{\sigma\in\Sigma_h} \sum_{m_1=0}^\ell \cdots \sum_{m_h=0}^\ell r_{m_{\sigma1}} \cdots r_{m_{\sigma h}} z_{m_{\sigma1}m_{\sigma2}} \cdots z_{m_{\sigma(h-1)}m_{\sigma h}}.
\]
Now the only effect of any fixed $\sigma$ on the inner $h$-fold sum is to permute the order of the indices; therefore setting $j_1=m_{\sigma_1}$, $j_2=m_{\sigma_2}$, and so on, we may write
\[
\Phi_h\bigg( \bigg( \sum_{i=0}^\ell r_i x_i \bigg)^h \bigg) = \frac1{h!} \sum_{\sigma\in\Sigma_h} \sum_{j_1=0}^\ell \cdots \sum_{j_h=0}^\ell r_{j_1} \cdots r_{j_h} z_{j_1j_2} \cdots z_{j_{h-1}j_h}.
\]
The inner $h$-fold sum no longer depends on $\sigma$, and so
\[
\Phi_h\bigg( \bigg( \sum_{i=0}^\ell r_i x_i \bigg)^h \bigg) = \sum_{j_1=0}^\ell \cdots \sum_{j_h=0}^\ell r_{j_1} \cdots r_{j_h} z_{j_1j_2} \cdots z_{j_{h-1}j_h} = \bigg( \sum_{j_1=0}^\ell \sum_{j_2=0}^\ell r_{j_1} r_{j_2} z_{j_1j_2} \bigg)^{h/2},
\]
which is equivalent to the statement of the lemma.
\end{proof}

\begin{remark}
\newcommand{\Sym}{\mathop{\rm Sym}}
The map $\Phi_h$ can also be interpreted as a rather natural $R$-module homomorphism from $\Sym^{2h}(M)$ to $\Sym^h(M\otimes_R M)$, where $M = R^{\oplus(\ell+1)} \cong S_1$. However, this interpretation does not seem to shorten the verification of the desired identity.
\end{remark}

We now wish to apply these results to a specific polynomial related to the moments of $\log G(n)$.
Given a real number $x$, let $X=(\log\log x)^{1/2}(\log\log\log x)^2$ as before, and let $\rho(X)$ denote the number of prime powers up to $X$. Define the polynomial
\[
Q(x_0, x_1, \ldots, x_{\rho(X)}) = \log 2 \cdot x_0 + \frac{1}{4}\sum_{i=1}^{\rho(X)} \Lambda(q_i)x_i^2.
\]
Note that the polynomial $P_n(x)$ defined in the introduction is equal to this polynomial $Q$ evaluated at the tuple $(x_0, x_1, \ldots, x_{\rho(X)}) = (\omega_0(n), \omega_{q_1}(n), \ldots, \omega_{q_{\rho(X)}}(n))$. For consistency, we will abuse notation and set $q_0 = 0$; this will be convenient when applying the results of this section to $P_n(x)$ in Section \ref{proof of moment asymptotics}.

Let $Q_i$ denote the partial derivative of $Q$ with respect to $x_i$. Observe that
\begin{align}\label{linearpart}
Q(x_0+y_0,\dots,x_{\rho(X)}+y_{\rho(X)}) - Q(y_0,\dots,y_{\rho(X)}) &= \log 2 \cdot x_0 + \frac{1}{2} \sum_{i=1}^{\rho(X)} \Lambda(q_i)x_i y_i + \sum_{i=1}^{\rho(X)} \Lambda(q_i)x_i^2 \nonumber \\ &= \sum_{i = 0}^{\rho(X)} x_iQ_i(y_0, \ldots, y_{\rho(X)}) + \sum_{i=1}^{\rho(X)} \Lambda(q_i)x_i^2.
\end{align}

\begin{definition} \label{Rh def}
Let $h$ be a positive integer. Define
\[
R_h(x_0,\dots,x_{\rho(X)},y_0,\dots,y_{\rho(X)}) = \big( Q(x_0+y_0,\dots,x_{\rho(X)}+y_{\rho(X)}) - Q(y_0,\dots,y_{\rho(X)}) \big)^h.
\]
To expand this out in gruesome detail, $R_h$ can be written as the sum of some number $B_h$ of monomials:
\begin{equation} \label{Rh expanded}
R_h(x_0,\dots,x_{\rho(X)},y_0,\dots,y_{\rho(X)}) = \sum_{\beta=1}^{B_h} r_{h\beta} \prod_{i=1}^{k_{h\beta}} x_{v(h,\beta, i)} \prod_{j=1}^{\tilde k_{h\beta}} y_{w(h,\beta, j)},
\end{equation}
where each $v(h,\beta, i)$ and $w(h,\beta, j)$ is an integer in $\{0, 1, \dots,{\rho(X)}\}$; the total $x$-degree of the $\beta$th monomial in the sum is $k_{h\beta}$, while its total $y$-degree is $\tilde k_{h\beta}$. From equation~\eqref{linearpart}, we see that each $k_{h\beta}$ is between $h$ and $2h$ (inclusive), each $\tilde k_{h\beta}$ is at most $h$, and each $k_{h\beta}+\tilde k_{h\beta}$ is also between $h$ and~$2h$.
\end{definition}

As it turns out, the most significant monomials on the right-hand side of equation~\eqref{Rh expanded} are those of minimal $x$-degree, that is, those monomials with $k_{h\beta}=h$. (These monomials will contribute to the main term of the calculation of the $h$th moment in Section~\ref{proof of moment asymptotics} when $h$ is even, while the other monomials contribute only to the error term.) Consequently we focus on these special monomials for the remainder of this section.

\begin{lemma} \label{Rh small part lemma}
The part of $R_h$ of total $x$-degree $h$ is
\begin{equation} \label{total x degree forms}
\sum_{\substack{\beta\le B_h \\ k_{h\beta}=h}} r_{h\beta} \prod_{i=1}^h x_{v(h,\beta, i)} \prod_{j=1}^{\tilde k_{h\beta}} y_{w(h,\beta,j)} = \bigg(\sum_{i = 0}^{\rho(X)} x_iQ_i(y_0, \ldots, y_{\rho(X)}) \bigg)^h.
\end{equation}
\end{lemma}

\begin{proof}
The left-hand side is exactly the definition of the part of $R_h$ of total $x$-degree $h$, or equivalently (since $k_{h\beta}\ge h$ always) the part of $R_h$ of total $x$-degree at most~$h$. But $R_h$ is the $h$th power of the polynomial $Q(x_0+y_0,\dots,x_{\rho(X)}+y_{\rho(X)}) - Q(y_1,\dots,y_{\rho(X)})$, whose part of total $x$-degree at most $1$ equals $\sum_{i = 0}^{\rho(X)} x_iQ_i(y_0, \ldots, y_{\rho(X)})$ by equation~\eqref{linearpart}.
\end{proof}

We are now ready to establish the proposition that will be used in Section~\ref{proof of moment asymptotics} when analyzing the main term of the even moments. For any positive even integer $h$, define
\begin{equation} \label{s_h def}
s_h = \frac{h!}{2^{h/2}(h/2)!}.
\end{equation}

\begin{prop} \label{Rh magic Phi prop}
Let $h$ be a positive even integer. In the notation of Definitions~\ref{Tk def} and~\ref{Rh def},
\begin{multline} \label{after applying Phi forms}
\frac1{(h/2)!} \sum_{\substack{\beta\le B_h \\ k_{h\beta}=h}} r_{h\beta} \prod_{j=1}^{\tilde k_{h\beta}} y_{w(h,\beta, j)} \sum_{\tau\in T_h} \prod_{i=1}^{h/2} z_{v(h,\beta,\Upsilon_1(i))v(h,\beta,\Upsilon_2(i))} \\
= s_h \bigg( \sum_{i = 0}^{\rho(X)} \sum_{j = 0}^{\rho(X)} Q_i(y_0, \ldots, y_{\rho(X)})Q_j(y_0, \ldots, y_{\rho(X)}) z_{ij} \bigg)^{h/2}.
\end{multline}
\end{prop}

\begin{proof}
Consider the operator $\Phi_h$ from Definition~\ref{Phih def}, using the ring $R = \R[y_0,\dots,y_{\rho(X)}]$. We establish the lemma by showing that the left- and right-hand sides of equation~\eqref{after applying Phi forms} are the results of applying $\Phi_h$ to $s_h$ times the left- and right-hand sides, respectively, of equation~\eqref{total x degree forms}.

Checking the right-hand side is easy: since $\Phi_h$ is an $R$-module homomorphism,
\begin{align*}
\Phi_h\bigg( s_h \bigg( \sum_{j=0}^{\rho(X)} x_j Q_j(y_1,\dots,y_{\rho(X)}) \bigg)^h \bigg) &= s_h \Phi_h\bigg( \bigg( \sum_{j=0}^{\rho(X)} Q_j(y_1,\dots,y_{\rho(X)}) x_j \bigg)^h \bigg) \\
&= s_h \bigg( \sum_{i=0}^{\rho(X)} \sum_{j=0}^{\rho(X)} Q_i(y_0,\dots,x_{\rho(X)}) Q_j(y_0,\dots,y_{\rho(X)}) z_{ij} \bigg)^{h/2}
\end{align*}
by Lemma~\ref{monomial function lemma} with $r_j = Q_j(y_1,\dots,y_{\rho(X)}) \in R$. As for the left-hand side: by $R$-linearity we have
\begin{align*}
\Phi_h \bigg( s_h \sum_{\substack{\beta\le B_h \\ k_{h\beta}=h}} r_{h\beta} & \prod_{i=1}^h x_{v(h,\beta, i)} \prod_{j=1}^{\tilde k_{h\beta}} y_{w(h,\beta,j)} \bigg) \\
&= s_h \sum_{\substack{\beta\le B_h \\ k_{h\beta}=h}} r_{h\beta} \prod_{j=1}^{\tilde k_{h\beta}} y_{w(h,\beta,j)} \Phi_h \bigg( \prod_{i=1}^h x_{v(h,\beta, i)} \bigg) \\
&= s_h \sum_{\substack{\beta\le B_h \\ k_{h\beta}=h}} r_{h\beta} \prod_{j=1}^{\tilde k_{h\beta}} y_{w(h,\beta,j)} \frac1{h!} \sum_{\sigma\in\Sigma_h} z_{v(h,\beta, \sigma1)v(h,\beta, \sigma2)} \cdots z_{v(h,\beta, \sigma(h-1))v(h,\beta, \sigma h)}.
\end{align*}
But by Lemma~\ref{Sigma to T lemma}, the set $\Sigma_h$ can be partitioned into subsets of size $2^{h/2}$, each subset corresponding to a particular $\tau\in T_K$ and consisting of those $\sigma$ for which equation~\eqref{Upsilons} holds. Therefore
\begin{align*}
\Phi_h \bigg( s_h & \sum_{\substack{\beta\le B_h \\ k_{h\beta}=h}} r_{h\beta}  \prod_{i=1}^h x_{v(h,\beta, i)} \prod_{j=1}^{\tilde k_{h\beta}} y_{w(h,\beta,j)} \bigg) \\
&= s_h \sum_{\substack{\beta\le B_h \\ k_{h\beta}=h}} r_{h\beta} \prod_{j=1}^{\tilde k_{h\beta}} y_{w(h,\beta,j)} \frac{2^{h/2}}{h!} \sum_{\tau\in T_h} z_{v(h,\beta, \Upsilon_1(1))v(h,\beta, \Upsilon_2(1))} \cdots z_{v(h,\beta, \Upsilon_1(h/2))v(h,\beta, \Upsilon_2(h/2))}.
\end{align*}
The lemma now follows upon noting that $s_h2^{h/2}/h! = 1/(\frac h2)!$.
\end{proof}

\section{Covariances of two additive functions}\label{term estimation}

The goal of this section is to evaluate certain expressions, arising from expanding the $h$th power of the right-hand side of equation~\eqref{hth moment}, in terms of certain ``covariances'' which we now define.

Throughout this section, $k$ is a fixed positive integer, $x>1$ is a real number, and $z = x^{1/2k}$.
For any two additive functions $g_1$ and $g_2$, define their \emph{covariance} to be
\[
\cov(g_1, g_2) = \cov(g_1, g_2;z) = \sum_{p \leq z} \frac{g_1(p)g_2(p)}{p} \left( 1 - \frac{1}{p} \right).
\]
Whenever $g_1(p),g_2(p) \ll \log p$ (as will be the case in our application), this definition can be simplified to
\begin{equation}  \label{cov with O(1)}
\cov(g_1, g_2) = \sum_{p \leq z} \frac{g_1(p)g_2(p)}{p} + O\bigg( \sum_{p \leq z} \frac{\log^2 p}{p^2} \bigg) = \sum_{p \leq z} \frac{g_1(p)g_2(p)}{p} + O(1).
\end{equation}
We begin by finding asymptotic formulas for these covariances when each of $g_1$ and $g_2$ is equal to one of the~$\omega_q$ ($q\ge0$).

The Bombieri--Vinogradov theorem will be an essential tool here and later in the paper; see for example~\cite[Theorem 17.1]{ik04} for the statement for the function $\psi(x;q,a)$, from which it is simple to derive the analogous versions for the functions $\theta(x;q,a)$ and $\pi(x;q,a)$ (an example of such a derivation is the proof of~\cite[Corollary 1.4]{AH}).

\begin{theorem}\label{bombvino}
For any positive real number $A$, there exists a positive real number $B=B(A)$ such that the estimates
\begin{align}
\sum_{2\le q \leq Q} \max_{(a, q) = 1} \left| \theta(x; q, a) - \frac{x}{\phi(q)} \right| &\ll_A \frac{x}{(\log x)^A} \label{theta BV} \\
\sum_{2\le q \leq Q} \max_{(a, q) = 1} \left| \pi(x; q, a) - \frac{\li(x)}{\phi(q)} \right| &\ll_A \frac{x}{(\log x)^A} \label{pi BV}
\end{align}
hold for all $x>1$, where $Q = x^{1/2}(\log x)^{-B}$.
\end{theorem}

The following three lemmas provide the desired evaluations of the relevant covariances; we must attend separately to the cases where neither, one, or both of the two additive functions equals~$\omega_0$.

\begin{lemma}  \label{little cov lemma}
Let $q_1$ and $q_2$ be powers of primes (possibly of the same prime), and let $[q_1, q_2]$ denote the least common multiple of $q_1$ and~$q_2$. Then
\[
\displaystyle \cov(\omega_{q_1}, \omega_{q_2}) = \frac{\log\log z}{\phi([q_1, q_2])} + O(1)
\]
uniformly for $q_1, q_2 \leq \sqrt z$.
\end{lemma}

\begin{proof}
Since each $\omega_{q_i}$ is uniformly bounded, and $\omega_{q_1}(p)\omega_{q_2}(p)=1$ precisely when $p$ is congruent to $1$ modulo $[q_1,q_2]$, equation~\eqref{cov with O(1)} becomes
\begin{align*}
\cov(\omega_{q_1}, \omega_{q_2}) = \sum_{\substack{p \leq z \\ p\equiv1\mod{[q_1,q_2]}}} \frac1{p} + O(1) 
= \frac{\log\log z}{\phi([q_1,q_2])} + O\bigg( \frac{\log [q_1,q_2]}{\phi([q_1,q_2])} \bigg) + O(1)
\end{align*}
by Lemma~\ref{mertenspartial}; and the first error term can be absorbed into the~$O(1)$.
\end{proof}

\begin{lemma}  \label{medium cov lemma}
If $q \leq z^{1/4}$ is a prime power, then
\[
\displaystyle \cov(\omega_{q}, \omega_{0}) = \frac{(\log\log z)^2}{2 \phi(q)} + O(\log\log z).
\]
\end{lemma}

\begin{proof}
We emulate the proof of~\cite[Lemma 2.1]{ep85}. In preparation for a partial summation calculation, we first show that
\begin{equation}  \label{q0 claim}
\sum_{p \leq t} \omega_q(p) \omega_0(p) = \sum_{\substack{p \leq t \\ p \equiv 1 \mod q}} \omega(p-1) = \frac{t\log\log t}{\phi(q)\log t} + O\bigg(\frac{t}{\log t}\bigg)
\end{equation}
for all $t>q$; the first equality follows from Definition~\ref{omega q def} of $\omega_q$ and Definition~\ref{omega0 def} of~$\omega_0$. When $q>\log^2 t$ this estimate is simple: the trivial bounds $\pi(t;q,1) < t/q$ and $\omega(p-1) \ll \log p$ result in
\[
\sum_{\substack{p \leq t \\ p \equiv 1 \mod q}} \omega(p-1) \ll \pi(t;q,1) \log t < \frac{t\log t}q \ll \frac t{\log t},
\]
which is consistent with the right-hand side of equation~\eqref{q0 claim} since $q>\log^2t$ implies that $(\log\log t)/ \phi(q) \ll (\log\log t \log\log\log t)/\log^2 t \ll 1$. Consequently, we may assume that $q \le \log^2t$.

Letting $\ell$ denote a variable of summation taking only prime values, and noting that at most $2$ primes greater than $p^{1/3}$ can divide $p-1$,
\begin{align}
\sum_{\substack{p \leq t \\ p \equiv 1 \mod q}} \omega(p-1) &= \sum_{\substack{p \leq t \\ p \equiv 1 \mod q}} \bigg( \sum_{\substack{\ell\mid(p-1) \\ \ell\le t^{1/3}}} 1 + \sum_{\substack{\ell\mid(p-1) \\ \ell> t^{1/3}}} 1 \bigg) \notag \\
&= \sum_{\substack{\ell \leq t^{1/3}}} \sum_{\substack{p \leq t \\ p \equiv 1 \mod q \\ p \equiv 1 \mod \ell}} 1 + O\bigg( \sum_{\substack{p \leq t \\ p \equiv 1 \mod q}} 2 \bigg) \notag \\
&= \bigg( \sum_{\substack{\ell \leq t^{1/3} \\ \ell \nmid q}} \sum_{\substack{p \leq t \\ p \equiv 1 \mod{q\ell}}} 1 + O\bigg( \sum_{\substack{\ell \leq t^{1/3} \\ \ell \mid q}} \sum_{\substack{p \leq t \\ p \equiv 1 \mod{q}}} 1 \bigg) \bigg) + O(\pi(t;q,1)) \notag \\
&= \sum_{\substack{\ell \leq t^{1/3} \\ \ell \nmid q}} \pi(t;q\ell,1) + O\big( \omega(q) \pi(t;q,1) \big) = \sum_{\substack{\ell \leq t^{1/3} \\ \ell \nmid q}} \pi(t;q\ell,1) + O\bigg( \frac t{\log t} \bigg),  \label{formerly S1 and S2}
\end{align}
where the last step follows from the Brun--Titchmarsh theorem (see~\cite[Theorem 3.9]{mv07}) and the assumption $q\le \log^2 t$:
\begin{equation}  \label{BT usage}
\omega(q) \pi(t;q,1) \ll \omega(q) \frac{t}{\phi(q)\log(t/q)} \ll \frac{\omega(q)}{\phi(q)} \frac{t}{\log t} \ll \frac t{\log t}.
\end{equation}
By the Bombieri--Vinogradov estimate~\eqref{pi BV} with $A=1$ (noting that every modulus $q\ell$ in the sum is at most $t^{1/3}\log^2 t \ll t^{1/2}(\log t)^{-B}$),
\begin{align}
\sum_{\substack{\ell \leq t^{1/3} \\ \ell \nmid q}} \pi(t;q\ell,1) &= \sum_{\substack{\ell \leq t^{1/3} \\ \ell \nmid q}} \frac{\li(t)}{\phi(q\ell)} + O\bigg( \sum_{\substack{\ell \leq t^{1/3} \\ \ell \nmid q}} \bigg| \pi(t;q\ell,1) - \frac{\li(t)}{\phi(q\ell)} \bigg| \bigg) \notag \\
&= \frac{\li(t)}{\phi(q)} \sum_{\substack{\ell \leq t^{1/3} \\ \ell \nmid q}} \frac1{\ell-1} + O\bigg(\frac{t}{\log t}\bigg) \label{BV use to copy} \\
&= \frac{\li(t)}{\phi(q)} \big( \log\log t^{1/3} + O(\omega(q)) \big) + O\bigg(\frac{t}{\log t}\bigg) = \frac{t\log\log t}{\phi(q)\log t} + O\bigg(\frac{t}{\log t}\bigg) \notag
\end{align}
by partial summation. Together with the estimate~\eqref{formerly S1 and S2}, this evaluation establishes the claim~\eqref{q0 claim}.

Define $S(t) = \sum_{p \leq t} \omega_q(p) \omega_0(p)$ to be the left-hand side of equation~\eqref{q0 claim}. Noting that $\omega_q(p) = 0$ for all $p \leq q$, we use partial summation to estimate
\begin{align*}
\cov(\omega_q, \omega_{0}) = \sum_{q<p \leq z} \omega_q(p) \omega_0(p) \frac{1}{p} = \int_q^z \frac{1}{t} \, dS(t) = \frac{S(z)}{z} + \int_q^z \frac{S(t)}{t^2} \, dt.
\end{align*}
By equation~\eqref{q0 claim}, the first term satisfies
\[
\frac{S(z)}{z} = \frac{\log\log z}{\phi(q) \log z} + O\bigg(\frac{1}{\log z}\bigg) \ll 1
\]
while
\begin{align*}
\int_q^z \frac{S(t)}{t^2} \, dt &= \int_q^z \bigg( \frac{\log\log t}{\phi(q)\, t \log t} + O\bigg(\frac{1}{t \log t}\bigg) \bigg) \, dt \\
&= \frac{(\log\log t)^2}{2 \phi(q)} \bigg|_q^z + O\big(\log\log t \big|_q^z\big) = \frac{(\log\log z)^2}{2 \phi(q)} + O(\log\log z).
\end{align*}
as required.
\end{proof}

\begin{lemma}  \label{big cov lemma}
For $z>2$,
\[
\displaystyle \cov(\omega_{0}, \omega_{0}) = \frac{(\log\log z)^3}{3} + O\big( (\log\log z)^2 \big).
\]
\end{lemma}

\begin{proof}
Now we emulate the proof of~\cite[Lemma 2.2]{ep85}. In preparation for a partial summation calculation, we first show that
\begin{equation}  \label{00 claim}
\sum_{p \leq t} \omega_0(p)^2 = \sum_{p \le t} \omega(p-1)^2 = \frac{t(\log\log t)^2}{\log t} + O\bigg(\frac{t | \!\log\log t|}{\log t}\bigg)
\end{equation}
for all $t>2$; again the first equality follows from Definition~\ref{omega0 def} of~$\omega_0$. Letting $\ell$ denote a variable of summation taking only prime values, and noting that at most $4$ primes greater than $p^{1/5}$ can divide $p-1$,
\begin{align*}
\sum_{p \le t} \omega(p-1)^2 &= \sum_{p \le t} \bigg( \sum_{\substack{\ell \mid (p-1) \\ \ell\le t^{1/5}}} 1 + O(4) \bigg)^2 \\
&= \sum_{p \le t} \sum_{\substack{\ell_1 \mid (p-1) \\ \ell\le t^{1/5}}} \sum_{\substack{\ell_2 \mid (p-1) \\ \ell\le t^{1/5}}} 1 + O\bigg( \sum_{p \le t} \sum_{\substack{\ell \mid (p-1) \\ \ell\le t^{1/5}}} 1 + \sum_{p \le t} 1 \bigg) \\
&= \sum_{\substack{\ell_1\le t^{1/5}}} \sum_{\substack{\ell_2\le t^{1/5}}} \sum_{\substack{p\le t \\ p\equiv1\mod{\ell_1} \\ p\equiv1\mod{\ell_2}}} 1 + O\bigg( \sum_{p \le t} \omega(p-1) + \pi(t) \bigg) \\
&= \sum_{\substack{\ell_1\le t^{1/5}}} \sum_{\substack{\ell_2\le t^{1/5} \\ \ell_2\ne\ell_1}} \pi(t;\ell_1\ell_2,1) + \sum_{\substack{\ell\le t^{1/5}}} \pi(t;\ell,1) + O\bigg( \frac{t|\!\log\log t|}{\log t} \bigg),
\end{align*}
where the error term in the last step was controlled using the $q=2$ case of equation~\eqref{q0 claim}.
Using the Bombieri--Vinogradov estimate~\eqref{pi BV} in a manner similar to the argument in equation~\eqref{BV use to copy} now yields
\begin{align*}
\sum_{p \le t} \omega(p-1)^2 &= \sum_{\substack{\ell_1\le t^{1/5}}} \sum_{\substack{\ell_2\le t^{1/5} \\ \ell_2\ne\ell_1}} \frac{\li(t)}{(\ell_1-1)(\ell_2-1)} + \sum_{\substack{\ell\le t^{1/5}}} \frac{\li(t)}{\ell-1} + O\bigg( \frac{t|\!\log\log t|}{\log t} \bigg) \\
&= \li(t) \bigg( \bigg( \sum_{\substack{\ell\le t^{1/5}}} \frac1{\ell-1} \bigg)^2 + O\bigg( \sum_{\substack{\ell\le t^{1/5}}} \frac1{\ell} \bigg) \bigg) + O\bigg( \frac{t|\!\log\log t|}{\log t} \bigg) \\
&= \li(t) \big( (\log\log t^{1/5})^2 + O(\log\log t) \big) + O\bigg( \frac{t|\!\log\log t|}{\log t} \bigg),
\end{align*}
which is enough to establish the claim~\eqref{00 claim}.

Define $S(t) = \sum_{p \leq t} \omega_0(p)^2$ to be the left-hand side of equation~\eqref{00 claim}, and again use partial summation to estimate
\begin{align*}
\cov(\omega_0, \omega_{0}) = \sum_{q<p \leq z} \omega_0(p)^2 \frac{1}{p} = \int_2^z \frac{1}{t} \, dS(t) = \frac{S(z)}{z} + \int_2^z \frac{S(t)}{t^2} \, dt.
\end{align*}
By equation~\eqref{00 claim}, the first term satisfies
\[
\frac{S(z)}{z} = \frac{(\log\log z)^2}{\log z} + O\bigg(\frac{|\!\log\log z|}{\log z}\bigg) \ll 1
\]
while
\begin{align*}
\int_2^z \frac{S(t)}{t^2} \, dt &= \int_2^z \bigg( \frac{(\log\log t)^2}{t \log t} + O\bigg(\frac{|\!\log\log t|}{t \log t}\bigg) \bigg) \, dt \\
&= \frac{(\log\log t)^3}{3} \bigg|_2^z + O\big( (\log\log t)^2 \big|_2^z\big) = \frac{(\log\log z)^3}3 + O\big( (\log\log z)^2 \big)
\end{align*}
as required.
\end{proof}

When we expand the $h$th power in the calculation of the moments $M_h(x)$ (as in equation~\eqref{two way}), we will need to estimate products of the additive functions $\omega_q$ ($q\ge0$) from Definitions~\ref{omega q def} and~\ref{omega0 def}, summed over many prime variables. Because of the presence of the multiplicative function $H$, defined momentarily, in these sums, it will be important how many distinct prime values are taken by these prime variables. The next three lemmas provide the details.

\begin{definition}  \label{H def}
Define a multiplicative function $H(n)$ by setting, for each prime power $p^\gamma$,
\begin{equation*}
H(p^\gamma) = \frac{1}p \bigg( 1-\frac{1}p \bigg)^\gamma + \bigg( {-}\frac{1}p \bigg)^\gamma \bigg( 1-\frac{1}p \bigg).
\end{equation*}
For any prime $p$, we note that $H(p^2) = \frac1p(1-\frac1p)$ and $H(p)=0$; in particular, $H(n)=0$ unless $n$ is squarefull. It is easy to check that $0\le H(p^\gamma) \le H(p^2)$ for every prime $p$ and every positive integer~$\gamma$.
\end{definition}

\begin{lemma} \label{paired off lemma}
Let $k$ be a positive even integer, and let $0\le\ell\le k$ be an integer. Suppose that $g_1=\cdots=g_\ell=\omega_0$, while the remaining functions $g_j$ ($\ell<j\le k)$ equal $\omega_{q_j}$ for some prime powers $q_j$. Then
\begin{multline} \label{paired off formula}
\sum_{\substack{p_1,\dots,p_k \leq z \\ p_1\cdots p_k \text{ squarefull} \\ \#\{p_1,\dots,p_k\} = k/2}} H(p_1\cdots p_k) g_1(p_1) \cdots g_k(p_k) \\
= \frac 1{(k/2)!} \sum_{\tau\in T_k} \prod_{j=1}^{k/2} \cov(g_{\Upsilon_1(j)},g_{\Upsilon_2(j)}) + O_k\big((\log\log x)^{(2\ell + k)/2 - 1}\big).
\end{multline}
\end{lemma}

\begin{proof}
All implicit constants in this proof may depend upon~$k$.
To each $k$-tuple $(p_1,\dots,p_k)$ counted by the sum on the left-hand side, we can uniquely associate a $(k/2)$-tuple $(q_1,\dots,q_{k/2})$ of primes satisfying $q_1 < \cdots < q_{k/2}$ such that each $q_j$ equals exactly two of the~$p_i$. This correspondence defines a unique $\tau\in T_k$, for which $\tau(i)$ equals the integer $j$ such that $p_i=q_j$. Therefore, by Definition~\ref{H def},
\begin{align*}
\sum_{\substack{p_1,\dots,p_k \leq z \\ p_1\cdots p_k \text{ squarefull} \\ \#\{p_1,\dots,p_k\} = k/2}} H(p_1\cdots p_k) & g_1(p_1) \cdots g_k(p_k) \\
&= \sum_{\tau\in T_k} \sum_{\substack{q_1 < \cdots < q_{k/2} \leq z}} H(q_1^2 \cdots q_{k/2}^2) g_1(q_{\tau(1)}) \cdots g_k(q_{\tau(k)}) \\
&= \frac1{(k/2)!} \sum_{\tau\in T_k} \sum_{\substack{q_1, \dots, q_{k/2} \leq z \\ q_1,\dots,q_{k/2} \text{ distinct}}} H(q_1^2 \cdots q_{k/2}^2) g_1(q_{\tau(1)}) \cdots g_k(q_{\tau(k)}) \\
&= \frac1{(k/2)!} \sum_{\tau\in T_k} \sum_{\substack{q_1, \dots, q_{k/2} \leq z \\ q_1,\dots,q_{k/2} \text{ distinct}}} g_1(q_{\tau(1)}) \cdots g_k(q_{\tau(k)}) \prod_{j=1}^{k/2} \frac{1}{q_j} \bigg( 1-\frac{1}{q_j} \bigg) \\
&= \frac1{(k/2)!} \sum_{\tau\in T_k} \sum_{\substack{q_1, \dots, q_{k/2} \leq z \\ q_1,\dots,q_{k/2} \text{ distinct}}} \prod_{j=1}^{k/2} g_{\Upsilon_1(j)}(q_j) g_{\Upsilon_2(j)}(q_j) \frac{1}{q_j} \bigg( 1-\frac{1}{q_j} \bigg).
\end{align*}
If we fix $\tau$ and $q_1,\dots,q_{k/2-1}$, the innermost sum over $q_{k/2}$ is
\begin{align*}
\sum_{\substack{q_{k/2} \leq z \\ q_{k/2} \notin \{q_1,\dots,q_{k/2-1}\}}} & g_{\Upsilon_1(k/2)}(q_{k/2}) g_{\Upsilon_2(k/2)}(q_{k/2}) \frac{1}{q_{k/2}} \bigg( 1-\frac{1}{q_{k/2}} \bigg) \\
&= \cov(g_{\Upsilon_1(k/2)},g_{\Upsilon_2(k/2)}) - \sum_{j=1}^{k/2-1} g_{\Upsilon_1(k/2)}(q_j) g_{\Upsilon_2(k/2)}(q_j) \frac{1}{q_j} \bigg( 1-\frac{1}{q_j} \bigg) \\
&= \cov(g_{\Upsilon_1(k/2)},g_{\Upsilon_2(k/2)}) + O(1),
\end{align*}
since each $g_i(q) \ll \log z$ for $q\le z$.
Summing in turn over $q_{k/2-1},\dots,q_1$ in the same way, we obtain
\begin{align*}
\sum_{\substack{p_1,\dots,p_k\leq z \\ p_1\cdots p_k \text{ squarefull} \\ \#\{p_1,\dots,p_k\} = k/2}} H( & p_1\cdots p_k) g_1(p_1) \cdots g_k(p_k) \\
&= \frac 1{(k/2)!} \sum_{\tau\in T_k} \prod_{j=1}^{k/2} \big( \cov(g_{\Upsilon_1(j)},g_{\Upsilon_2(j)}) + O(1) \big).
\end{align*}
Upon multiplying out the product corresponding to some $\tau \in T_k$, we obtain the leading term
\[
\prod_{j=1}^{k/2} \cov(g_{\Upsilon_1(j)},g_{\Upsilon_2(j)})
\]
together with terms that involve at most $k/2-1$ covariances. An examination of Lemmas~\ref{little cov lemma}--\ref{big cov lemma} reveals that the order of magnitude of the leading term (as a function of $z$) is $(\log\log z)^{(2\ell + k)/2}$, regardless of how the $g_j$ are paired with one another by~$\tau$, and that every non-leading term is $\ll (\log\log z)^{(2\ell + k)/2-1}$ uniformly in the possibilities for the $g_j$, We conclude that
\begin{align*}
\sum_{\substack{p_1,\dots,p_k \leq z \\ p_1\cdots p_k \text{ squarefull} \\ \#\{p_1,\dots,p_k\} = k/2}} H( & p_1\cdots p_k) g_1(p_1) \cdots g_k(p_k) \\
&= \frac 1{(k/2)!} \sum_{\tau\in T_k} \prod_{j=1}^{k/2} \cov(g_{\Upsilon_1(j)},g_{\Upsilon_2(j)}) + O_k((\log\log x)^{(2\ell + k)/2 - 1})
\end{align*}
as desired (where we have used $z = x^{1/2k}$ in the error term).
\end{proof}

\begin{lemma}
\label{one omega0 is big lemma}
Let $k$ be a positive integer, and let $g_1,\dots,g_k$ be functions satisfying $g_1(p),\dots,g_k(p) \ll \log p$. Then for any $1\le i\le k$,
\begin{align*}
\sum_{\substack{p_1,\dots,p_k \leq z \\ \omega_0(p_i) > 4k\log\log z}} H(p_1\cdots p_k) g_1(p_1) \cdots g_k(p_k) \ll_k \frac1{(\log z)^{1/2}}.
\end{align*}
\end{lemma}

\begin{proof}
All implicit constants in this proof may depend upon~$k$.
Suppose that $q_1,\cdots,q_s$ are the distinct primes such that $\{p_1,\dots,p_k\} = \{q_1,\cdots,q_s\}$, and let $m$ denote any integer such that $q_m=p_i$.
From Definition~\ref{H def}, we know that $0\le H(p_1\cdots p_k) \le H(q_1^2\cdots q_s^2) \le 1/q_1\cdots q_s$. Therefore, from the hypothesis on the sizes of the~$g_j(p)$,
\begin{align}
\sum_{\substack{p_1,\dots,p_k \leq z \\ \omega_0(p_i) > 4k\log\log z}} & H(p_1\cdots p_k) g_1(p_1) \cdots g_k(p_k) \notag \\
&\ll_k (\log z)^k \sum_{s=1}^k \sum_{m=1}^s \sum_{\substack{q_1,\dots,q_s \leq z \\ \omega_0(q_m) > 4k\log\log z}} \frac1{q_1\cdots q_s} \notag \\
&\ll_k (\log z)^k \sum_{s=1}^k \sum_{m=1}^s \bigg( \sum_{\substack{q_m \le z \\ \omega_0(q_m) > 4k\log\log z}} \frac1{q_m} \bigg) \prod_{\substack{1\le i\le k \\ i\ne m}} \sum_{q_i\le z} \frac1{q_i} \notag \\
&\ll_k (\log z)^k \sum_{s=1}^k (\log\log z)^s \sum_{m=1}^s \sum_{\substack{q_m \le z \\ \omega_0(q_m) > 4k\log\log z}} \frac1{q_m}  \label{pre Nicolas}
\end{align}
by Mertens's theorem. Note that
\[
\sum_{\substack{q_m \le z \\ \omega_0(q_m) > 4k\log\log z}} \frac1{q_m} = \sum_{\substack{q_m \le z \\ \omega(q_m-1) > 4k\log\log z}} \frac1{q_m} \le \sum_{\substack{n \le z \\ \omega(n) > 4k\log\log z}} \frac1n.
\]
A result of Erd{\H o}s and Nicolas~\cite{Erd1978-1979} implies that the number of $n \leq x$ satisfying $\omega(n) > 4k \log\log x$ is $\ll {x}/{(\log x)^{1 + 4k\log4k - 4k}}$; partial summation then implies that the right-hand sum is $\ll 1/(\log x)^{4k\log4k - 4k}$. Equation~\eqref{pre Nicolas} therefore implies
\[
\sum_{\substack{p_1,\dots,p_k \leq z \\ \omega_0(p_i) > 4k\log\log z}} H(p_1\cdots p_k) g_1(p_1) \cdots g_k(p_k) \ll_k (\log z)^k (\log\log z)^k \frac1{(\log z)^{4k\log 4k-4k}},
\]
and the lemma follows from the fact that $4k\log4k-5k>\frac12$ for $k\ge1$.
\end{proof}

\begin{lemma}
\label{unpaired off lemma}
Let $k$ be a positive integer, and let $0\le\ell\le k$ be an integer. Suppose that $g_1=\cdots=g_\ell=\omega_0$, while the remaining functions $g_j$ ($\ell<j\le k)$  equal $\omega_{q_j}$ for some prime powers $q_j$. When $k$ is even,
\begin{equation*}
\sum_{\substack{p_1,\dots,p_k \leq z \\ p_1\cdots p_k \text{ squarefull} \\ \#\{p_1,\dots,p_k\} < k/2}} H(p_1\cdots p_k) g_1(p_1) \cdots g_k(p_k) \ll_k (\log\log x)^{(2\ell + k)/2-1},
\end{equation*}
while when $k$ is odd,
\begin{equation*}
\sum_{\substack{p_1,\dots,p_k \leq z \\ p_1\cdots p_k \text{ squarefull} \\ \#\{p_1,\dots,p_k\} < k/2}} H(p_1\cdots p_k) g_1(p_1) \cdots g_k(p_k) \ll_k (\log\log x)^{(2\ell + k - 1)/2}.
\end{equation*}
\end{lemma}

\noindent We remark that when $k$ is odd, the condition of summation $\#\{p_1,\dots,p_k\} < k/2$ is always satisfied; we have nevertheless included the condition, for later convenience.

\begin{proof}
All implicit constants in this proof may depend upon~$k$. We begin by noting that by Lemma~\ref{one omega0 is big lemma}, it suffices to consider the sum on the left-hand side with the extra summation condition $\max \omega_0(p_i) \leq 4k\log\log z$ inserted.

To each $k$-tuple $(p_1,\dots,p_k)$ counted by the sum on the left-hand side, we associate the positive integer $s = \#\{p_1,\dots,p_k\}$, the primes $q_1<\cdots<q_s$ such that $\{q_1,\dots,q_s\} = \{p_1,\dots,p_k\}$, and the integers $\gamma_1,\dots,\gamma_s\ge2$ such that $q_j$ equals exactly $\gamma_j$ of the~$p_i$; note that $\bgamma = (\gamma_1,\dots,\gamma_s)$ is a composition, not a partition, of $k$, since we are not assuming any monotonicity of the $\gamma_j$. Let $T_\bgamma$ denote the set of functions from $\{1,\dots,k\}$ to $\{1,\dots,s\}$ such that for each $1\le j\le s$, exactly $\gamma_j$ elements of $\{1,\dots,k\}$ are mapped to~$j$. Given any $\tau\in T_\bgamma$, define $\Upsilon_1(j)$ and $\Upsilon_2(j)$ to be two distinct preimages of $j$ in $\{1,\dots,k\}$; we will never need to know exactly which two preimages or to distinguish between the two. Finally, for any such $\tau$, define $\ell'$ to be the number of functions among $g_{\Upsilon_1(1)},g_{\Upsilon_2(1)},\dots,g_{\Upsilon_1(s)},g_{\Upsilon_2(s)}$ that equal $\omega_0$, and set $M_\tau = (4k\log\log z)^{\ell-\ell'}$.

First, observe that
\begin{align*}
\sum_{\substack{p_1,\dots,p_k \leq z \\ p_1\cdots p_k \text{ squarefull} \\ \#\{p_1,\dots,p_k\} < k/2 \\ \max \omega_0(p_i) \le 4k\log\log z}} & H(p_1\cdots p_k) g_1(p_1) \cdots g_k(p_k) \\
&= \sum_{1\le s<k/2} \sum_{\substack{q_1 < \cdots < q_s \le z \\ \max \omega_0(q_i) \le 4k\log\log z}} \sum_{\substack{\gamma_1,\dots,\gamma_s\ge2 \\ \gamma_1+\cdots+\gamma_s=k}} H(q_1^{\gamma_1}\cdots q_s^{\gamma_s}) \sum_{\tau\in T_\bgamma} g_1(q_{\tau(1)}) \cdots g_k(q_{\tau(k)}).
\end{align*}
By Definition~\ref{H def}, we may bound $H(q_1^{\gamma_1}\cdots q_s^{\gamma_s})$ by $H(q_1^2\cdots q_s^2)$. Moreover, in the innermost summand, we retain all of the factors of the form $g_{\Upsilon_1(j)}(q_j)$ and $g_{\Upsilon_2(j)}(q_j)$ while bounding all of the other $g_i(q_{\tau(i)})$ by their pointwise upper bounds, which results in a factor of $M_\tau$:
\begin{align*}
\sum_{\substack{p_1,\dots,p_k \leq z \\ p_1\cdots p_k \text{ squarefull} \\ \#\{p_1,\dots,p_k\} < k/2 \\ \max \omega_0(p_i) \le 4k\log\log z}} & H(p_1\cdots p_k) g_1(p_1) \cdots g_k(p_k) \\
&\le \sum_{1\le s<k/2} \sum_{\substack{q_1 < \cdots < q_s \le z}} \sum_{\substack{\gamma_1,\dots,\gamma_s\ge2 \\ \gamma_1+\cdots+\gamma_s=k}} H(q_1^2\cdots q_s^2) \sum_{\tau\in T_\bgamma} M_\tau \prod_{j=1}^s g_{\Upsilon_1(j)}(q_j) g_{\Upsilon_2(j)}(q_j) \\
&= \sum_{1\le s<k/2} \frac1{s!} \sum_{\substack{ q_1,\dots,q_s \le z}} \sum_{\substack{\gamma_1,\dots,\gamma_s\ge2 \\ \gamma_1+\cdots+\gamma_s=k}} \sum_{\tau\in T_\bgamma} M_\tau \prod_{j=1}^s g_{\Upsilon_1(j)}(q_j) g_{\Upsilon_2(j)}(q_j) \frac{1}{q_j} \bigg( 1 - \frac{1}{q_j} \bigg).
\end{align*}
Moving the sum over the $q_j$ to the inside and summing, we obtain
\begin{align*}
\sum_{\substack{p_1,\dots,p_k \leq z \\ p_1\cdots p_k \text{ squarefull} \\ \#\{p_1,\dots,p_k\} < k/2 \\ \max \omega_0(p_i) \le 4k\log\log z}} H( & p_1\cdots p_k) g_1(p_1) \cdots g_k(p_k) \\
&\le \sum_{1\le s<k/2} \frac{1}{s!} \sum_{\substack{\gamma_1,\dots,\gamma_s\ge2 \\ \gamma_1+\cdots+\gamma_s=k}} \sum_{\tau\in T_\bgamma} M_\tau \prod_{j=1}^s \cov(g_{\Upsilon_1(j)},g_{\Upsilon_2(j)}).
\end{align*}
An examination of Lemmas~\ref{little cov lemma}--\ref{big cov lemma} reveals that each product on the right-hand side is $\ll (\log\log z)^{s + \ell'}$ regardless of how the $g_j$ are paired with one another by~$\tau$; consequently,
\begin{align*}
M_\tau \prod_{j=1}^s \cov(g_{\Upsilon_1(j)},g_{\Upsilon_2(j)}) &\ll_k (\log\log z)^{\ell - \ell'}(\log\log z)^{s + \ell'} \\
&= (\log\log z)^{s + \ell} \le \begin{cases}
(\log\log x)^{k/2 - 1 + \ell}, &\text{if $k$ is even}, \\
(\log\log x)^{(k-1)/2 + \ell}, &\text{if $k$ is odd}.
\end{cases}
\end{align*}
The lemma follows upon summing over $\tau$, the $\gamma_i$ and $s$, which results in a constant that depends only on $k$.
\end{proof}

We are now ready to establish the main result of this section, which will be used repeatedly in Section~\ref{proof of moment asymptotics}. Recall that the functions $f_p$ and $F_g$ were defined in Definition~\ref{fF def}.

\begin{prop}  \label{summing products of k Fs}
Let $k$ be a positive even integer, and let $0\le\ell\le k$ be an integer. Suppose that $g_1=\cdots=g_\ell=\omega_0$, while the remaining functions $g_j$ ($\ell<j\le k)$ equal $\omega_{q_j}$ for some prime powers $q_j$. When $k$ is even,
\[
\sum_{n \leq x} \prod_{j=1}^k F_{g_j}(n) = \frac x{(k/2)!} \sum_{\tau\in T_k} \prod_{j=1}^{k/2} \cov(g_{\Upsilon_1(j)},g_{\Upsilon_2(j)}) + O_k\big( x(\log\log x)^{(2\ell + k)/2 - 1})\big),
\]
while when $k$ is odd,
\[
\sum_{n \leq x} \prod_{j=1}^k F_{g_j}(n) \ll_k x (\log\log x)^{(2\ell + k - 1)/2}.
\]
\end{prop}


\begin{proof}
All implicit constants in this proof may depend upon~$k$.
Expanding out the left-hand side using Definition~\ref{fF def} results in
\begin{align*}
\sum_{n \leq x} \prod_{j=1}^k F_{g_j}(n) &= \sum_{n \leq x} \prod_{j=1}^k \sum_{p\leq z} g_j(p)f_p(n) \\
&= \sum_{p_1,\dots,p_k \leq z} g_1(p_1) \cdots g_k(p_k) \sum_{n \leq x} f_{p_1\cdots p_k}(n) \\
&= \sum_{p_1,\dots,p_k \leq z} g_1(p_1) \cdots g_k(p_k) \big( H(p_1\cdots p_k)x + O\big( 2^{\omega(p_1\cdots p_k)} \big) \big),
\end{align*}
where the last equality follows from~\cite[equation before equation~(9)]{gs07} with a slight change of notation.
Each $\omega_q$ is bounded by $1$, while $\omega_0(p) = \omega(p-1)$ is trivially bounded by $\log p/\log 2$; in particular, $2g_j(p_j) \ll \log z$ for all $p_j \le z$. Therefore
\begin{align*}
\sum_{p_1,\dots,p_k \leq z} g_1(p_1) \cdots g_k(p_k) 2^{\omega(p_1\cdots p_k)} \ll \sum_{p_1,\dots,p_k \leq z} \prod_{j=1}^k \log z = (\pi(z)\log z)^k \ll z^k = \sqrt x,
\end{align*}
and so
\begin{align*}
\sum_{n \leq x} \prod_{j=1}^k F_{g_j}(n) = x \sum_{p_1,\dots,p_k \leq z} H(p_1\cdots p_k) g_1(p_1) \cdots g_k(p_k) + O(\sqrt x).
\end{align*}
Since $H(p_1\cdots p_k)$ vanishes unless $p_1\cdots p_k$ is squarefull by Definition~\ref{H def}, there are at most $k/2$ distinct primes among $p_1,\dots,p_k$, and so we can write
\begin{multline*}
\sum_{p_1,\dots,p_k \leq z} H(p_1\cdots p_k) g_1(p_1) \cdots g_k(p_k) = \sum_{\substack{p_1,\dots,p_k \leq z \\ p_1\cdots p_k \text{ squarefull} \\ \#\{p_1,\dots,p_k\} = k/2}} H(p_1\cdots p_k) g_1(p_1) \cdots g_k(p_k) \\
+ \sum_{\substack{p_1,\dots,p_k \leq z \\ p_1\cdots p_k \text{ squarefull} \\ \#\{p_1,\dots,p_k\} < k/2}} H(p_1\cdots p_k) g_1(p_1) \cdots g_k(p_k).
\end{multline*}
The proposition now follows upon appealing to Lemmas~\ref{paired off lemma} and~\ref{unpaired off lemma}.
\end{proof}

\section{Calculating the moments}\label{proof of moment asymptotics}

We are now ready to carry out, for $h \geq 1$, the computation of the moments $M_h(x)$. In particular, the proof of Proposition~\ref{moment asymptotics} requires some preparatory work, which we organize into Lemmas~\ref{matching main terms lemma}--\ref{lower order terms}.
We also find an asymptotic formula for the function $D(x)$ in Proposition~\ref{constant in mean}; together with Lemmas~\ref{infinitesums} and~\ref{OOM of sum of partial derivs}, this calculation reveals the origins of the perhaps mysterious constants $A$, $B$, and $C$ appearing in Theorem~\ref{main theorem}. Finally, we proof Proposition~\ref{moment asymptotics} at the end of this section.

Recall that $X = (\log\log x)^{1/2}(\log\log\log x)^2$, a notation that will persist throughout this section; we shall always assume that $X\ge2$.
As our starting point, we define
\begin{equation} \label{S1 S2 def}
S_1 = \sum_{2\le q \leq X} \Lambda(q) F_{\omega_{q}}(n)^2 \quad\text{and}\quad S_2 = \sum_{2\le q \leq X} 2 \Lambda(q) \mu(\omega_{q}) F_{\omega_{q}}(n)
\end{equation}
and use equations~\eqref{Mh def} and~\eqref{hth moment} to write
\begin{equation} \label{two way}
\begin{split}
M_h(x) = \sum_{n \leq x} \big( P_n(x) - D(x) \big)^h &= \sum_{n \leq x} \big( \log 2 \cdot F_{\omega_0}(n) + S_1 + S_2 \big)^h \\
&= \sum_{\substack{h_0,h_1,h_2\ge0 \\ h_0+h_1+h_2=h}} \binom h{h_0,h_1,h_2} \sum_{n \leq x} \big( \log 2 \cdot F_{\omega_0}(n) \big)^{h_0} S_1^{h_1}S_2^{h_2},
\end{split}
\end{equation}
where the $\binom h{h_0,h_1,h_2}$ are multinomial coefficients.
Since $\mu(\omega_q)$ is large and positive while $F_{\omega_q}$ is an oscillatory function, and $F_{\omega_0}$ is significantly larger on average than any $F_{\omega_q}$ with $q\ge2$, our intuition should be that the largest summands on the right-hand side correspond to $h_1=0$. Indeed, the following lemma gives an alternate expression for the sum of these large summands, in a notation that will allow us to apply our work from Section~\ref{polynomial arithmetic}.
Recall that, for convenience, we set $q_0 = 0$ (so that $\omega_{q_0} = \omega_0$).

\begin{lemma}\label{matching main terms lemma}
Let $h$ be a positive integer. In the notation of Definition~\ref{Rh def} and equation~\eqref{S1 S2 def},
\begin{align*}
\sum_{\substack{h_0,h_1,h_2\ge0 \\ h_0+h_1+h_2=h \\ h_1=0}} \binom h{h_0,h_1,h_2} \big( \log 2 \cdot F_{\omega_0}(n) \big)^{h_0} S_2^{h_2}  &= \sum_{h_0=0}^h \binom h{h_0} \sum_{n \leq x} \big( \log 2 \cdot F_{\omega_0}(n) \big)^{h_0} S_2^{h-h_0} \\
&=
\sum_{\substack{\beta\le B_h \\ k_{h\beta}=h}} r_{h\beta} \prod_{j=1}^{\tilde k_{h\beta}} \mu(\omega_{q_{w(h,\beta, j)}}) \sum_{n \leq x} \prod_{i=1}^h F_{\omega_{q_{v(h,\beta, i)}}}(n).
\end{align*} 
\end{lemma}

\begin{proof}
The first equality is a simple change of variables, so we focus on the second equality.
Since $\omega_q(n) = \mu(\omega_q) + F_{\omega_q}(n)$ by equation~\eqref{function equals mean plus F}, the formulas~\eqref{Pn D new notation} can be combined as
\begin{align*}
P_n(x) - D(x) &= \log 2 \cdot \big( \omega_0(n) - \mu(\omega_q) \big) + \frac14 \sum_{2\le q\le X} \big( \omega_{q}(n)^2 - \mu(\omega_q)^2 \big) \Lambda(q) \\
&= \log 2 \cdot F_{\omega_0}(n) + \frac14 \sum_{2\le q\le X} \big( 2\mu(\omega_q)F_{\omega_q}(n) + F_{\omega_q}(n)^2 \big) \Lambda(q) \\
&= \log 2 \cdot F_{\omega_{q_0}}(n) + \frac{1}{2} \sum_{i=1}^{\rho(X)} \Lambda({q_i}) F_{\omega_{q_i}}(n) \mu(\omega_{q_i}) + \sum_{i=1}^{\rho(X)} \Lambda({q_i}) F_{\omega_{q_i}}(n)^2 \\
&= Q\big( F_{\omega_{q_0}}(n)+\mu(\omega_{q_0}),\dots,F_{\omega_{q_{\rho(X)}}}(n)+\mu(\omega_{q_{\rho(X)}}) \big) - Q\big( \mu(\omega_{q_0}),\dots,\mu(\omega_{q_{\rho(X)}}) \big)
\end{align*}
by comparison to equation~\eqref{linearpart}.
Therefore, by equation~\eqref{Mh def} and Definition~\ref{Rh def},
\begin{align}
M_h(x) = \sum_{n \leq x} \big( P_n(x) - D(x) \big)^h &= \sum_{n \leq x} R_h(F_{\omega_0}(n), \ldots, F_{\omega_{q_\ell}}(n), \mu(\omega_0), \ldots, \mu(\omega_{q_\ell})) \notag \\
 &= \sum_{n \leq x} \sum_{\beta=1}^{B_h} r_{h\beta} \prod_{i=1}^{k_{h\beta}} F_{\omega_{q_{v(h,\beta, i)}}}(n) \prod_{j=1}^{\tilde k_{h\beta}} \mu(\omega_{q_{w(h,\beta, j)}}) \notag \\
 &= \sum_{\beta=1}^{B_h} r_{h\beta} \prod_{j=1}^{\tilde k_{h\beta}} \mu(\omega_{q_{w(h,\beta, j)}}) \sum_{n \leq x} \prod_{i=1}^{k_{h\beta}} F_{\omega_{q_{v(h,\beta, i)}}}(n). \label{grislyexpansion}
\end{align}
Note that each monomial on the right-hand side has $k_{h\beta}$ factors of the form $F_{\omega_q}$ for various $0\le q\le X$.

On the other hand, if we insert the definitions~\eqref{S1 S2 def} into the right-hand side of equation~\eqref{two way} and expand out the powers $S_1^{h_1}S_2^{h_2}$, each resulting monomial will have $h_0+2h_1+h_2$ factors of the form $F_{\omega_q}$. Therefore, for any integer $h\le m\le 2h$,
\[
\sum_{\substack{h_0,h_1,h_2\ge0 \\ h_0+h_1+h_2=h \\ h_0+2h_1+h_2=m}} \binom h{h_0,h_1,h_2} \sum_{n \leq x} \big( \log 2 \cdot F_{\omega_0}(n) \big)^{h_0} S_1^{h_1}S_2^{h_2}
= 
\sum_{\substack{\beta\le B_h \\ k_{h\beta}=m}} r_{h\beta} \prod_{j=1}^{\tilde k_{h\beta}} \mu(\omega_{q_{w(h,\beta, j)}}) \sum_{n \leq x} \prod_{i=1}^{k_{h\beta}} F_{\omega_{q_{v(h,\beta, i)}}}(n).
\]
In particular, $h_0+2h_1+h_2=m$ in these sums precisely when $h_1=0$, so the $m=h$ case of the above equation is equivalent to the statement of the lemma.
\end{proof}

The following preliminary lemma estimates a sum that appears more than once in the proof of Lemma \ref{lower order terms} below. For the remainder of this section, all implicit constants may depend upon $h$, $h_0$, $h_1$, and $h_2$.

\begin{lemma}\label{sumlambdamu}
For any nonnegative integers $h_1$ and $h_2$,
\begin{align*}
\sum_{2\le q_1, \ldots, q_{h_1 + h_2} \leq X} \prod_{i = 1}^{h_1 + h_2} \Lambda(q_i) \prod_{i = h_1 + 1}^{h_1 + h_2} \mu(\omega_{q_i}) \ll (\log\log x)^{(h_1 + 2h_2)/2}(\log\log\log x)^{2h_1 + h_2}.
\end{align*}
\end{lemma}

\begin{proof}
We sum on each $q_i$ separately. For each $1 \leq i \leq h_1$, we simply have
\[
\sum_{2\le q_i \leq X} \Lambda(q_i) \ll X = (\log\log x)^{1/2}(\log\log\log x)^2
\]
by the prime number theorem, giving a total contribution of $(\log\log x)^{h_1/2}(\log\log\log x)^{2h_1}$. On the other hand, when $h_1 + 1 \leq i \leq h_1 + h_2$, equation~(\ref{muomegaq}) gives
\[
\mu(\omega_{q_i}) \ll \frac{\log\log x}{\phi(q_i)}
\]
since $q\le X<\log x$.
Therefore, for each $h_1 + 1 \leq i \leq h_1 + h_2$,
\[
\sum_{2\le q_i \leq X} \Lambda(q_i) \mu(\omega_{q_i}) \ll \log\log x \sum_{2\le q_i \leq X} \frac{\Lambda(q_i)}{\phi(q_i)} \ll \log\log x \cdot \log\log\log x
\]
by partial summation, giving a total contribution of $(\log\log x \cdot \log\log\log x)^{h_2}$. Collecting exponents yields the lemma.
\end{proof}

We now handle all the terms on the right-hand side of equation~\eqref{two way} when $h$ is odd, and the lower-order terms in the case when $h$ is even, with the following lemma. We do so by brute-force expansion of the $h$th power and using the results of Section \ref{term estimation}.

\begin{lemma}\label{lower order terms}
Let $h_0$, $h_1$, and $h_2$ be nonnegative integers, and set $h=h_0+h_1+h_2$. Suppose that either $h$ is odd, or $h$ is even and $h_1\ne0$. Then with $S_1$ and $S_2$ defined as in equation~\eqref{S1 S2 def},
\[
\sum_{n \leq x} ( \log 2 \cdot F_{\omega_0}(n))^{h_0} S_1^{h_1}S_2^{h_2} \ll x(\log\log x)^{3h/2-1/4}(\log\log\log x)^{2h}
\]
for $x\ge e^{e^3}$.
\end{lemma}

\begin{proof}
Since
\begin{align*}
S_1^{h_1} &= \bigg( \sum_{2\le q \leq X} \Lambda(q) F_{\omega_{q}}(n)^2 \bigg)^{h_1} = \sum_{2\le q_1, \ldots, q_{h_1} \leq X} \prod_{i = 1}^{h_1} \Lambda(q_i) F_{\omega_{q_i}}(n)^2
\end{align*}
and
\begin{align*}
S_2^{h_2} &= \bigg( \sum_{2\le q \leq X} 2 \Lambda(q)  \mu(\omega_{q}) F_{\omega_{q}}(n) \bigg)^{h_2} \ll \sum_{2\le q_1, \ldots, q_{h_2} \leq X} \prod_{i = 1}^{h_2}  \Lambda(q_i)  \mu(\omega_{q_i})F_{\omega_{q_i}}(n),
\end{align*}
the sum under consideration satisfies
\begin{align} \label{nastysum}
\sum_{n \leq x} & ( \log 2 \cdot F_{\omega_0}(n))^{h_0} S_1^{h_1}S_2^{h_2} \\
&\ll \sum_{n \leq x} \sum_{2\le q_1, \ldots, q_{h_1 + h_2} \leq X} F_{\omega_0}(n)^{h_0} \prod_{i = 1}^{h_1 + h_2} \Lambda(q_i) \prod_{i = h_1 + 1}^{h_1 + h_2} \mu(\omega_{q_i}) \prod_{i = 1}^{h_1} F_{\omega_{q_i}}(n)^2 \prod_{i = h_1 + 1}^{h_1 + h_2} F_{\omega_{q_i}}(n) \notag \\
&= \sum_{2\le q_1, \ldots, q_{h_1 + h_2} \leq X} \prod_{i = 1}^{h_1 + h_2} \Lambda(q_i) \prod_{i = h_1 + 1}^{h_1 + h_2} \mu(\omega_{q_i}) \bigg( \sum_{n \leq x} F_{\omega_0}(n)^{h_0} \prod_{i = 1}^{h_1} F_{\omega_{q_i}}(n)^2 \prod_{i = h_1 + 1}^{h_1 + h_2} F_{\omega_{q_i}}(n) \bigg).\notag
\end{align}

We will consider two cases, depending on the parity of $h_0+h_2$; when $h_0+h_2$ is even, we additionally assume that $h_1\ne0$. A moment's thought verifies that these two cases do exhaust the possibilities for $h_0$, $h_1$, and~$h_2$.

\medskip\noindent \textbf{Case 1: $h_0 + h_2$ is odd.}
In the inner sum on the right-hand side of equation~\eqref{nastysum},
each summand is the product of $h_0 + 2h_1 + h_2$ values of $F$-functions. By Proposition~\ref{summing products of k Fs} with $\ell=h_0$ and $k=h_0 + 2h_1 + h_2$ (which is odd),
\begin{align*}
\sum_{n \leq x} F_{\omega_0}(n)^{h_0} \prod_{i = 1}^{h_1} F_{\omega_{q_i}}(n)^2 \prod_{i = h_1 + 1}^{h_1 + h_2} F_{\omega_{q_i}}(n) \ll x (\log\log x)^{(3h_0 + 2h_1 + h_2 - 1)/2}.
\end{align*}
Inserting this upper bound into the right-hand side of equation~\eqref{nastysum} yields
\begin{align*}
\sum_{n \leq x} \big( \log 2 \cdot F_{\omega_0}(n) \big)^{h_0} & S_1^{h_1}S_2^{h_2} \\
&\ll x (\log\log x)^{(3h_0 + 2h_1 + h_2 - 1)/2} \sum_{2\le q_1, \ldots, q_{h_1 + h_2} \leq X} \prod_{i = 1}^{h_1 + h_2} \Lambda(q_i) \prod_{i = h_1 + 1}^{h_1 + h_2} \mu(\omega_{q_i}) \\
&\ll x (\log\log x)^{(3h_0 + 2h_1 + h_2 - 1)/2} \cdot (\log\log x)^{(h_1 + 2h_2)/2}(\log\log\log x)^{2h_1 + h_2} \\
&\le x (\log\log x)^{(3h - 1)/2} (\log\log\log x)^{2h}
\end{align*}
by Lemma~\ref{sumlambdamu} (since $h=h_0+h_1+h_2$), which establishes the lemma in this case.

\medskip\noindent \textbf{Case 2: $h_0+ h_2$ is even and $h_1\ne0$.}
In the inner sum on the right-hand side of equation~\eqref{nastysum},
each summand is again the product of $h_0 + 2h_1 + h_2$ values of $F$-functions. By Proposition~\ref{summing products of k Fs} with $\ell=h_0$ and $k=h_0 + 2h_1 + h_2$ (which is now even),
\begin{multline*}
\sum_{n \leq x} F_{\omega_0}(n)^{h_0} \prod_{i = 1}^{h_1} F_{\omega_{q_i}}(n)^2 \prod_{i = h_1 + 1}^{h_1 + h_2} F_{\omega_{q_i}}(n) \\
\ll x \sum_{\tau \in T_k} \prod_{j = 1}^{k/2} \cov(\omega_{q_{\Upsilon_1(j)}}, \omega_{q_{\Upsilon_2(j)}}) + x(\log\log x)^{(3h_0 + 2h_1 + h_2)/2-1}.
\end{multline*}
Inserting this upper bound into the right-hand side of equation~\eqref{nastysum} yields
\begin{align}
\sum_{n \leq x} \big( \log 2 & {}\cdot F_{\omega_0}(n) \big)^{h_0} S_1^{h_1}S_2^{h_2} \notag \\
&\ll x \sum_{2\le q_1, \ldots, q_{h_1 + h_2} \leq X} \prod_{i = 1}^{h_1 + h_2} \Lambda(q_i) \prod_{i = h_1 + 1}^{h_1 + h_2} \mu(\omega_{q_i}) \sum_{\tau \in T_k} \prod_{j = 1}^{k/2} \cov(\omega_{q_{\Upsilon_1(j)}}, \omega_{q_{\Upsilon_2(j)}}) \notag \\
&\qquad{}+ x (\log\log x)^{(3h_0 + 2h_1 + h_2)/2 - 1} \sum_{2\le q_1, \ldots, q_{h_1 + h_2} \leq X} \prod_{i = 1}^{h_1 + h_2} \Lambda(q_i) \prod_{i = h_1 + 1}^{h_1 + h_2} \mu(\omega_{q_i}) \notag \\
&\ll x \sum_{\tau \in T_k} \sum_{2\le q_1, \ldots, q_{h_1 + h_2} \leq X} \prod_{i = 1}^{h_1 + h_2} \Lambda(q_i) \prod_{i = h_1 + 1}^{h_1 + h_2} \mu(\omega_{q_i}) \prod_{j = 1}^{k/2} \cov(\omega_{q_{\Upsilon_1(j)}}, \omega_{q_{\Upsilon_2(j)}}) \notag \\
&\qquad{}+ x (\log\log x)^{3h/2 - 1} (\log\log\log x)^{2h}  \label{key cov product}
\end{align}
by Lemma~\ref{sumlambdamu} and an examination of exponents similar to the end of the proof of Case~1.

Now, by Lemmas~\ref{little cov lemma}--\ref{big cov lemma}, for any $0\le q,q'\le X$ we have
\begin{equation}  \label{cov cases}
\cov(\omega_{q}, \omega_{q'}) \ll \begin{cases}
(\log\log x)/\phi(q), & \text{if } q,q'\ge 2, \\
(\log\log x)^2/\phi(q), & \text{if } q\ge2 \text{ and } q'=0, \\
(\log\log x)^2/\phi(q'), & \text{if } q'\ge2 \text{ and } q=0, \\
(\log\log x)^3, & \text{if } q=q'=0.
\end{cases}
\end{equation}
Notice that the first upper bound is, intentionally, crude in general: by Lemma~\ref{little cov lemma}, we could divide not just by $\phi(q)$ but by $\phi([q,q'])$. However, $\phi([q,q'])$ can be as small as $\phi(q)$ in the worst case (when $q'$ divides $q$). Fortunately, our argument will succeed even with this worst-case assumption. (We have also bounded $\log\log z$ above by $\log\log x$, which is fairly insignificant.)

For a given $\tau\in T_k$ (which is a two-to-one function), define $\Delta(\tau)$ to be a subset of $\{1,\dots,h_1\}$ of size at least $h_1/2$ such that $\tau$ is one-to-one when restricted to~$\Delta(\tau)$. When we use the upper bounds~\eqref{cov cases} in the innermost product on the right-hand side of equation~\eqref{key cov product}, the resulting estimate will include a factor of $\prod_{i\in\Delta(\tau)} 1/\phi(q_i)$.
Furthermore, the resulting exponent of $\log\log x$ is $k/2+h_0 = (3h_0+2h_1+h_2)/2$, regardless of how the $g_j$ are paired with one another by~$\tau$. In other words, equation~\eqref{key cov product} becomes
\begin{align}  \label{Lee's trick!}
\sum_{n \leq x} & \big( \log 2 {}\cdot F_{\omega_0}(n) \big)^{h_0} S_1^{h_1}S_2^{h_2} \\
&\ll x (\log\log x)^{(3h_0+2h_1+h_2)/2} \sum_{\tau \in T_k} \sum_{2\le q_1, \ldots, q_{h_1 + h_2} \leq X} \prod_{i\in\Delta(\tau)} \frac{\Lambda(q_i)}{\phi(q_i)} \prod_{\substack{1\le i\le h_1 + h_2 \\ i\notin\Delta(\tau)}} \Lambda(q_i) \prod_{i = h_1 + 1}^{h_1 + h_2} \mu(\omega_{q_i}) \notag \\
&\qquad{}+ x (\log\log x)^{3h/2 - 1} (\log\log\log x)^{2h}. \notag
\end{align}
We now sum on each $q_i$ separately (still fixing $\tau$ for the moment), in a similar manner to the proof of Lemma~\ref{sumlambdamu}. For each $1 \leq i \leq h_1$ such that $i\notin\Delta(\tau)$, the prime number theorem gives
\[
\sum_{2\le q_i \leq X} \Lambda(q_i) \ll X = (\log\log x)^{1/2}(\log\log\log x)^2
\]
resulting in a total contribution of $(\log\log x)^{(h_1-\#\Delta(\tau))/2}(\log\log\log x)^{2(h_1-\#\Delta(\tau))}$. On the other hand, for each $1 \leq i \leq h_1$ such that $i\in\Delta(\tau)$, partial summation gives
\[
\sum_{2\le q_i \leq X} \frac{\Lambda(q_i)}{\phi(q_i)} \ll \log X \ll \log\log\log x,
\]
resulting in a total contribution of $(\log\log\log x)^{\#\Delta(\tau)}$.
Lastly, when $h_1 + 1 \leq i \leq h_1 + h_2$, equation~(\ref{muomegaq}) gives
\[
\sum_{2\le q_i \leq X} \Lambda(q_i) \mu(\omega_{q_i}) \ll \log\log x \sum_{2\le q_i \leq X} \frac{\Lambda(q_i)}{\phi(q_i)} \ll \log\log x \cdot \log\log\log x
\]
as above, resulting in a total contribution of $(\log\log x \cdot \log\log\log x)^{h_2}$. The product of all these contributions is
\begin{align*}
(\log\log x)^{(h_1-\#\Delta(\tau))/2+h_2} (\log\log\log x)^{2h_1+h_2-\#\Delta(\tau)} &\le (\log\log x)^{h_1/4+h_2} (\log\log\log x)^{2h} \\
&\le (\log\log x)^{h_1/2+h_2-1/4} (\log\log\log x)^{2h},
\end{align*}
where we have used the assumption $h_1>0$ in the last inequality.

Finally we insert this estimate back into equation~\eqref{Lee's trick!}, which gives
\begin{align*}
\sum_{n \leq x} \big( \log 2 {}\cdot F_{\omega_0}(n) & \big)^{h_0} S_1^{h_1}S_2^{h_2} \\
&\ll x (\log\log x)^{(3h_0+2h_1+h_2)/2} \sum_{\tau \in T_k} (\log\log x)^{h_1/2+h_2-1/4} (\log\log\log x)^{2h} \\
&\qquad{}+ x (\log\log x)^{3h/2 - 1} (\log\log\log x)^{2h} \\
&\ll x (\log\log x)^{3h/2-1/4} (\log\log\log x)^{2h}
\end{align*}
(the sum over $\tau$ can now be ignored, since the implicit constant may depend upon~$h$), which completes the proof of the lemma.
\end{proof}

Two particular sums of arithmetic functions will arise in the evaluation of the main term for $M_h(x)$; we asymptotically evaluate those sums in the following lemma, after which we give an asymptotic formula for the ``mean'' $D(x)$ that appears in the definition of~$M_h(x)$.

\begin{lemma}\label{infinitesums}
Recall from Theorem \ref{main theorem} that
\[
A_0 = \frac{1}{4} \sum_{p} \frac{p^2 \log p}{(p - 1)^3(p + 1)} \quad\text{and}\quad B = \frac14 \sum_p \frac{p^3(p^4 - p^3 - p^3 - p - 1)(\log p)^2}{(p - 1)^6(p + 1)^2(p^2 + p + 1)}.
\]
When $X>2$,
\begin{itemize}
\item[(a)] $\displaystyle \frac{1}{4}\sum_{2 \leq q \leq X} \frac{\Lambda(q)}{\phi(q)^2} = A_0 + O\bigg( \frac{1}{X} \bigg)$;
\item[(b)] $\displaystyle \frac{1}{4} \sum_{2 \leq q_1 \leq X}\sum_{2 \leq  q_2 \leq X} \frac{\Lambda(q_1)\Lambda(q_2)}{\phi(q_1)\phi(q_2)\phi([q_1, q_2])} = 4A_0^2 + B + O\bigg(\frac{\log X}{X}\bigg)$.
\end{itemize}
\end{lemma}

\begin{proof}
(a) We need only observe that
\begin{equation}  \label{getting 4A_0}
\sum_{q \geq 2} \frac{\Lambda(q)}{\phi(q)^2} = \sum_{p} \sum_{j = 1}^\infty \frac{\Lambda(p^j)}{\phi(p^j)^2} = \sum_{p} \frac{\log p}{(p-1)^2} \sum_{j = 1}^\infty \frac{1}{(p^{j-1})^2} = \sum_{p} \frac{\log p}{(p-1)^2} \frac{p^2}{p^2-1} = 4A_0,
\end{equation}
while partial summation bounds the tail of this convergent series by
\[
\frac{1}{4} \sum_{q > X} \frac{\Lambda(q)}{\phi(q)^2} \ll \frac{1}{X}.
\]

(b) If $q_1 = p_1^r$ and $q_2 = p_2^s$ with $p_1 \neq p_2$, then $[q_1, q_2] = p_1^r p_2^s$; on the other hand, if $q_1 = p^r$ and $q_2 = p^s$ are powers of the same prime, then $[q_1, q_2] = p^{\max(r, s)}$. Therefore,
\begin{align}\label{varianceconstantlastpart}
\sum_{q_1\ge2}\sum_{q_2\ge2} & \frac{\Lambda(q_1)\Lambda(q_2)}{\phi(q_1)\phi(q_2)\phi([q_1, q_2])}\nonumber \\
&= \sum_{p_1^r} \sum_{\substack{p_2^s \\ p_2\ne p_1}} \frac{(\log p_1)(\log p_2)}{p_1^{r-1}(p_1-1)p_2^{s-1}(p_2-1)p_1^{r-1}(p_1-1)p_2^{s-1}(p_2-1)}  \nonumber \\
&\qquad{}+ \sum_p \sum_{r=1}^\infty \sum_{s=1}^\infty \frac{(\log p)^2}{p^{r-1}(p-1)p^{s-1}(p-1)p^{\max(r, s)-1}(p-1)} \nonumber \\
&= \bigg( \sum_{p_1^r} \sum_{\substack{p_2^s}}  \frac{(\log p_1)(\log p_2)}{p_1^{r-1}(p_1-1)p_2^{s-1}(p_2-1)p_1^{r-1}(p_1-1)p_2^{s-1}(p_2-1)}   \\
&\qquad{} - \sum_p \sum_{r=1}^\infty \sum_{s=1}^\infty \frac{(\log p)(\log p)}{p^{r-1}(p-1)p^{s-1}(p-1)p^{r-1}(p-1)p^{s-1}(p-1)} \bigg) \nonumber \\
&\qquad{}+ \sum_p \sum_{r=1}^\infty \sum_{s=1}^\infty \frac{(\log p)^2}{p^{r-1}(p-1)p^{s-1}(p-1)p^{\max(r, s)-1}(p-1)}. \nonumber
\end{align}
By equation~\eqref{getting 4A_0}, the double sum on the right-hand side is simply
\[
\bigg( \sum_{p^j} \frac{\log p}{(p^{j-1})^2(p-1)^2} \bigg)^2 = (4A_0)^2.
\]
On the other hand, using the power series identities
\[
\sum_{r=1}^\infty \sum_{s=1}^\infty x^r x^s x^r x^s = \bigg( \frac{x^2}{1-x^2} \bigg)^2
\]
and
\begin{align*}
\sum_{r=1}^\infty \sum_{s=1}^\infty x^r x^s x^{\max\{r,s\}} &= 2 \sum_{r=1}^\infty \sum_{s=r}^\infty x^r x^s x^s - \sum_{r=1}^\infty x^r x^r x^r \\
&= 2 \sum_{r=1}^\infty \frac{x^{3r}}{1 - x^2} - \sum_{r=1}^\infty x^{3r} = \frac{1+x^2}{1-x^2} \frac{x^3}{1-x^3},
\end{align*}
we may evaluate the pair of triple sums on the right-hand side of equation~\eqref{varianceconstantlastpart} as
\begin{equation*}
\sum_p \bigg( {-} \frac{p^4(\log p)^2}{(p-1)^4} \bigg( \frac{p^{-2}}{1-p^{-2}} \bigg)^2 + \frac{p^3(\log p)^2} {(p-1)^3} \frac{1+p^{-2}}{1-p^{-2}} \frac{p^{-3}}{1-p^{-3}} \bigg) = 4B.
\end{equation*}
Thus equation~\eqref{varianceconstantlastpart} simplifies to
\[
\frac{1}{4} \sum_{q_1\ge2}\sum_{q_2\ge2} \frac{\Lambda(q_1)\Lambda(q_2)}{\phi(q_1)\phi(q_2)\phi([q_1, q_2])} = 4A_0^2 + B,
\]
and it therefore suffices to show that the tail
\begin{multline}
\sum_{q_1\ge2}\sum_{q_2\ge2} \frac{\Lambda(q_1)\Lambda(q_2)}{\phi(q_1)\phi(q_2)\phi([q_1, q_2])} - \sum_{2 \leq q_1 \leq X}\sum_{2 \leq q_2 \leq X} \frac{\Lambda(q_1)\Lambda(q_2)}{\phi(q_1)\phi(q_2)\phi([q_1, q_2])} \\
= \sum_{q_1>X}\sum_{q_2>X} \frac{\Lambda(q_1)\Lambda(q_2)}{\phi(q_1)\phi(q_2)\phi([q_1, q_2])} + 2\sum_{2 \leq q_1 \leq X} \sum_{q_2 > X} \frac{\Lambda(q_1)\Lambda(q_2)}{\phi(q_1)\phi(q_2)\phi([q_1, q_2])}  \label{double sum tail}
\end{multline}
is $\ll (\log X)/X$.

Since $\phi([q_1, q_2]) \geq \phi(q_2)$, the second sum on the right-hand side can be bounded crudely:
\[
2 \sum_{2 \leq q_1 \leq X} \sum_{q_2 > X} \frac{\Lambda(q_1)\Lambda(q_2)}{\phi(q_1)\phi(q_2)^2} = 2 \bigg( \sum_{2 \leq q_1 \leq X} \frac{\Lambda(q_1)}{\phi(q_1)} \bigg) \bigg( \sum_{q_2 > X} \frac{\Lambda(q_2)}{\phi(q_2)^2} \bigg) \ll {\log X} \cdot \frac1{X}.
\]
by partial summation. Finally, we handle the first sum on the right-hand side of equation~\eqref{double sum tail} by splitting it as
\begin{align*}
\sum_{q_1>X}\sum_{q_2>X} & \frac{\Lambda(q_1)\Lambda(q_2)}{\phi(q_1)\phi(q_2)\phi([q_1, q_2])} \\
&= \sum_{p_1^r>X} \sum_{\substack{p_2^s>X \\ p_2\ne p_1}} \frac{(\log p_1)(\log p_2)}{\phi(p_1^r)\phi(p_2^s)\phi(p_1^r p_2^s)} + \sum_p \sum_{\substack{r \\ p^r > X}} \sum_{\substack{s \\ p^s > X}} \frac{(\log p)^2}{\phi(p^r)\phi(p^s)\phi(p^{\max(r, s)})} \\
&\le \sum_{p_1^r>X} \sum_{\substack{p_2^s>X \\ p_2\ne p_1}} \frac{(\log p_1)(\log p_2)}{\phi(p_1^r)^2\phi(p_2^s)^2} + 2 \sum_p \sum_{\substack{r \\ p^r > X}} \sum_{\substack{s=r}}^\infty \frac{(\log p)^2}{\phi(p^r)\phi(p^s)\phi(p^s)} \\
&\ll \bigg( \sum_{p^r > X} \frac{\log p}{p^{2r}} \bigg)^2 + \sum_{p^r>X} \frac{(\log p)^2}{p^r} \sum_{s=r}^\infty \frac1{p^{2s}} \\
&\ll \bigg( \frac1X \bigg)^2 + \sum_{p^r>X} \frac{(\log p)^2}{p^{3r}} \ll \frac{\log X}{X^2}
\end{align*}
by partial summation.
\end{proof}

\begin{prop}\label{constant in mean}
Recall that $D(x)$ was defined in equation~\eqref{Pn D new notation}. When $x>e^e$,
\[
D(x) = A(\log\log x)^2 + O\bigg(\frac{(\log\log x)^{3/2}}{(\log\log\log x)^2} \bigg),
\]
where $A = \displaystyle \frac{\log 2}{2} + A_0$ as in Theorem~\ref{main theorem}.
\end{prop}

\begin{proof}
We begin by establishing asymptotics for $\mu(\omega_0)$ and $\mu(\omega_q)$ for $q \leq X$. First,
\[
\mu(\omega_0) = \sum_{p \leq x} \frac{\omega_0(p)}{p} = \sum_{p \leq x} \frac{\omega_2(p)\omega_0(p)}{p} = \cov(\omega_2,\omega_0)+O(1) = \frac{1}{2}(\log\log x)^2 + O(\log\log x)
\]
by Lemma~\ref{medium cov lemma} (with $x$ in place of~$z$). On the other hand,
\begin{align}\label{muomegaq}
\mu(\omega_q) = \sum_{p \leq x} \frac{\omega_q(p)}{p} = \sum_{\substack{p \leq x \\ p\equiv1\mod q}} \frac{1}{p} = \frac{\log\log x}{\phi(q)} + O\bigg( \frac{\log q}{\phi(q)} \bigg)
\end{align}
by Lemma \ref{mertenspartial}. Squaring, and using the fact that $\log q = o(\log\log x)$ for $q$ in the range of summation, yields
\begin{align*} 
\mu(\omega_q)^2 = \frac{(\log\log x)^2}{\phi(q)^2} + O\bigg(\frac{\log q}{\phi(q)^2} \log\log x \bigg).
\end{align*}

Inserting these estimates for $\mu(\omega_0)$ and $\mu(\omega_q)$ into equation~(\ref{Pn D new notation}), we obtain
\begin{align*}
D(x) &= \frac{\log 2}{2} (\log\log x)^2 + O(\log\log x) + \frac{1}{4} \sum_{2\le q \leq X} \Lambda(q) \bigg( \frac{(\log\log x)^2}{\phi(q)^2} + O\bigg(\frac{\log q}{\phi(q)^2} \log\log x \bigg) \bigg) \\
&= \frac{\log 2}{2} (\log\log x)^2 + \frac{1}{4}(\log\log x)^2 \sum_{2\le q \leq X} \frac{\Lambda(q)}{\phi(q)^2} + O\bigg(\log\log x \sum_{2\le q \leq X} \frac{\Lambda(q) \log q}{\phi(q)^2} \bigg) \\
&= \bigg( \frac{\log 2}{2} + \frac{1}{4}\sum_{2\le q \leq X} \frac{\Lambda(q)}{\phi(q)^2} \bigg) (\log\log x)^2 + O(\log\log x)
\end{align*}
by partial summation.
By Lemma~\ref{infinitesums}, we may replace the coefficient of $(\log\log x)^2$ by $A + O(1/X)$, obtaining
\[
D(x) = A(\log\log x)^2 + O\bigg(\frac{(\log\log x)^{3/2}}{(\log\log\log x)^2} \bigg),
\]
as claimed.
\end{proof}

The following lemma gives the asymptotic size of a double sum that will appear shortly in the proof of Proposition~\ref{moment asymptotics}.

\begin{lemma}\label{OOM of sum of partial derivs}
Recall from Theorem~\ref{main theorem} that $C = \frac{(\log 2)^2}{3} + 2 A_0 \log 2 + 4A_0^2 + B$. We have
\begin{multline}  \label{secondmoment}
\sum_{i=0}^{\rho(X)} \sum_{j=0}^{\rho(X)} Q_i \big( \mu(\omega_{q_0}),\dots,\mu(\omega_{q_{\rho(X)}}) \big) Q_j\big( \mu(\omega_{q_0}),\dots,\mu(\omega_{q_{\rho(X)}}) \big) \cov(\omega_{q_i},\omega_{q_j}) \\
= C (\log\log x)^3 + O\bigg(\frac{(\log\log x)^{5/2}}{\log\log\log x}\bigg).
\end{multline} 
\end{lemma}

\begin{proof}
We begin by computing the partial derivatives appearing in the double sum. We have
\[
Q_0 \big( \mu(\omega_{q_0}),\dots,\mu(\omega_{q_{\rho(X)}}) \big) = \log 2
\]
and, for $i \neq 0$, we use equation~(\ref{muomegaq}) to write
\[
Q_i \big( \mu(\omega_{q_0}),\dots,\mu(\omega_{q_{\rho(X)}}) \big) = \frac{1}{2}\Lambda(q_i)\mu(\omega_{q_i}) = \frac{1}{2} \frac{\Lambda(q_i)}{\phi(q_i)} \log\log x + O\bigg(\frac{\Lambda(q_i) \log q_i}{\phi(q_i)} \bigg).
\]
So, by Lemma~\ref{big cov lemma}, the summand on the left-hand side of equation~\eqref{secondmoment} corresponding to $i = j = 0$ is of the form
\begin{align*}
Q_0 \big( \mu(\omega_{q_0}),\dots,\mu(\omega_{q_{\rho(X)}}) \big) & Q_0 \big( \mu(\omega_{q_0}),\dots,\mu(\omega_{q_{\rho(X)}}) \big) \cov(\omega_{q_0}, \omega_{q_0}) \\
&= \frac{(\log 2)^2}{3}(\log\log x)^3 + O((\log\log x)^2);
\end{align*}
similarly, by Lemma~\ref{medium cov lemma} the summands corresponding to $i = 0$ and $j \neq 0$ are of the form
\begin{align*}
Q_0 \big( \mu(\omega_{q_0}),\dots,\mu(\omega_{q_{\rho(X)}}) \big) & Q_j \big( \mu(\omega_{q_0}),\dots,\mu(\omega_{q_{\rho(X)}}) \big) \cov(\omega_{q_0}, \omega_{q_j}) \\
&= \frac{\log 2}{4}\frac{\Lambda(q_j)}{\phi(q_j)^2}(\log\log x)^3 + O\bigg(\frac{\Lambda(q_i)(\log\log x)^2}{\phi(q_j)^2}\bigg).
\end{align*}
and the summands corresponding to $i \neq 0$ and $j = 0$ are the same up to labeling. Finally, by Lemma~\ref{little cov lemma} we have that the summands on the left-hand side of equation~\eqref{secondmoment} corresponding to $i \neq 0$ and $j \neq 0$ are of the form
\begin{align*}
Q_i \big( \mu(\omega_{q_0}),\dots,\mu(\omega_{q_{\rho(X)}}) \big) & Q_j \big( \mu(\omega_{q_0}),\dots,\mu(\omega_{q_{\rho(X)}}) \big) \cov(\omega_{q_i}, \omega_{q_j}) \\
&= \frac{1}{4}\frac{\Lambda(q_i)\Lambda(q_j)}{\phi(q_i)\phi(q_j)\phi([q_i, q_j])}(\log\log x)^3 + O\bigg(\frac{\Lambda(q_i)\Lambda(q_j)}{\phi(q_i)\phi(q_j)}(\log\log x)^2\bigg).
\end{align*}
Combining these last three evaluations results in
\begin{align}
\sum_{i=0}^{\rho(X)} \sum_{j=0}^{\rho(X)} & Q_i \big( \mu(\omega_{q_0}),\dots,\mu(\omega_{q_{\rho(X)}}) \big) Q_j\big( \mu(\omega_{q_0}),\dots,\mu(\omega_{q_{\rho(X)}}) \big) \cov(\omega_{q_i},\omega_{q_j}) \nonumber \\
&= \frac{(\log 2)^2}{3}(\log\log x)^3 + O((\log\log x)^2) \nonumber \\
&\qquad{}+ 2\sum_{j = 1}^{\rho(X)} \bigg( \frac{\log 2}{4}\frac{\Lambda(q_j)}{\phi(q_j)^2}(\log\log x)^3 + O\bigg(\frac{\Lambda(q_i)(\log\log x)^2}{\phi(q_j)^2}\bigg) \bigg) \nonumber \\
&\qquad{}+ \sum_{i = 1}^{\rho(X)} \sum_{j = 1}^{\rho(X)} \bigg( \frac{1}{4}\frac{\Lambda(q_i)\Lambda(q_j)}{\phi(q_i)\phi(q_j)\phi([q_i, q_j])}(\log\log x)^3 + O\bigg(\frac{\Lambda(q_i)\Lambda(q_j)}{\phi(q_i)\phi(q_j)}(\log\log x)^2\bigg)\bigg) \notag \\
&= \bigg( \frac{(\log 2)^2}{3} + \frac{\log 2}{2} \sum_{i = 1}^{\rho(X)} \frac{\Lambda(q_i)}{\phi(q_j)^2} + \frac{1}{4} \sum_{i = 1}^{\rho(X)} \sum_{j = 1}^{\rho(X)} \frac{\Lambda(q_i)\Lambda(q_j)}{\phi(q_i)\phi(q_j)\phi([q_i, q_j])} \bigg) (\log\log x)^3 \notag \\
&\qquad{}+ O((\log\log x)^2(\log\log\log x)^2)
\label{moment double sum messy coeff}
\end{align}
by partial summation.
By Lemma \ref{infinitesums}, the coefficient of $(\log\log x)^3$ above is equal to
\[
\frac{(\log 2)^2}{3} + 2\log 2 A_0 + 4A_0^2 + B + O\bigg( \frac{1}{(\log\log x)^{1/2}\log\log\log x} \bigg).
\]
Inserting this expression into equation~\eqref{moment double sum messy coeff} finishes the proof.
\end{proof}

We now have all of the auxiliary results needed to carry out the asymptotic evaluation of the moments~$M_h(x)$.

\begin{proof}[Proof of Proposition~\ref{moment asymptotics}]
We start from equation~\eqref{two way}:
\begin{equation*}
M_h(x) = \sum_{\substack{h_0,h_1,h_2\ge0 \\ h_0+h_1+h_2=h}} \binom h{h_0,h_1,h_2} \sum_{n \leq x} \big( \log 2 \cdot F_{\omega_0}(n) \big)^{h_0} S_1^{h_1}S_2^{h_2}.
\end{equation*}
If $h\ge1$ is odd, then Lemma~\ref{lower order terms} applies to every inner sum, yielding
\begin{align*}
M_h(x) &\ll \sum_{\substack{h_0,h_1,h_2\ge0 \\ h_0+h_1+h_2=h}} \binom h{h_0,h_1,h_2} x(\log\log x)^{3h/2 - 1/4}(\log\log\log x)^{2h} \\
&\ll x(\log\log x)^{3h/2 - 1/4}(\log\log\log x)^{2h},
\end{align*}
since the implicit constant may depend upon~$h$.
In particular, $M_h(x) = o\big( x(\log\log x)^{3h/2} \big)$ for each odd $h$, as required.

On the other hand, if $h\ge2$ is even, then Lemma~\ref{lower order terms} applies to all summands except those for which $h_1=0$, so that
\begin{equation} \label{other hand}
\begin{split}
M_h(x) &= \sum_{\substack{h_0,h_1,h_2\ge0 \\ h_0+h_1+h_2=h \\ h_1=0}} \binom h{h_0,h_1,h_2} \big( \log 2 \cdot F_{\omega_0}(n) \big)^{h_0} S_2^{h_2} +O \big( x(\log\log x)^{3h/2 - 1/4}(\log\log\log x)^{2h} \big) \\
&= \sum_{\substack{\beta\le B_h \\ k_{h\beta}=h}} r_{h\beta} \prod_{j=1}^{\tilde k_{h\beta}} \mu(\omega_{q_{w(h,\beta, j)}}) \sum_{n \leq x} \prod_{i=1}^h F_{\omega_{q_{v(h,\beta, i)}}}(n) +O \big( x(\log\log x)^{3h/2 - 1/4}(\log\log\log x)^{2h} \big)
\end{split}
\end{equation}
by Lemma~\ref{matching main terms lemma}. Notice, in this translation of notation, that factors of the form $\mu(\omega_q)$ on the right-hand side all arise from the term $S_2^{h_2}$; in particular, $\tilde k_{h\beta} = h_2$, and each ${q_{w(h,\beta, j)}}$ is a prime power not exceeding $X$ (rather than~$0$), so that $\mu(\omega_{q_{w(h,\beta, j)}}) \ll \log\log x$ by equation~\eqref{muomegaq}. Similarly, of the factors of the form $F_{\omega_q}$, we see that $h_0$ of them are $F_{\omega_0}$, while the other $h_2$ are of the form $F_{\omega_q}$ for prime powers~$q$.
Therefore, in the main term in equation~\eqref{other hand}, we may apply Proposition~\ref{summing products of k Fs} with $\ell=h_0$ and $k=h=h_0+h_2$ to obtain
\begin{align}
\sum_{\substack{\beta\le B_h \\ k_{h\beta}=h}} & r_{h\beta} \prod_{j=1}^{\tilde k_{h\beta}} \mu(\omega_{q_{w(h,\beta, j)}}) \sum_{n \leq x} \prod_{i=1}^h F_{\omega_{q_{v(h,\beta, i)}}}(n) \notag \\
&= \sum_{\substack{\beta\le B_h \\ k_{h\beta}=h}} r_{h\beta} \prod_{j=1}^{\tilde k_{h\beta}} \mu(\omega_{q_{w(h,\beta, j)}}) \bigg( \frac x{(h/2)!} \sum_{\tau\in T_h} \prod_{i=1}^{h/2} \cov(\omega_{q_{v(h,\beta,\Upsilon_1(i))}},\omega_{q_{v(h,\beta,\Upsilon_2(i))}}) + O( x(\log\log x)^{(2h_0+h)/2 - 1} ) \bigg) \notag \\
&= \frac x{(h/2)!} \sum_{\substack{\beta\le B_h \\ k_{h\beta}=h}} r_{h\beta} \prod_{j=1}^{\tilde k_{h\beta}} \mu(\omega_{q_{w(h,\beta, j)}}) \sum_{\tau\in T_h} \prod_{i=1}^{h/2} \cov(\omega_{q_{v(h,\beta,\Upsilon_1(i))}},\omega_{q_{v(h,\beta,\Upsilon_2(i))}}) \notag \\
&\qquad{}+ O( (\log\log x)^{h_2} \cdot x(\log\log x)^{(3h_0+h_2)/2 - 1} );  \label{for this main term}
\end{align}
note that this last error term is exactly $x(\log\log x)^{3h/2-1}$.
By Proposition~\ref{Rh magic Phi prop} with $y_j=\mu(\omega_{q_j})$ and $z_{ij}=\cov(\omega_{q_i},\omega_{q_j})$,
\begin{align*}
\frac x{(h/2)!} & \sum_{\substack{\beta\le B_h \\ k_{h\beta}=h}} r_{h\beta} \prod_{j=1}^{\tilde k_{h\beta}} \mu(\omega_{q_{w(h,\beta, j)}}) \sum_{\tau\in T_h} \prod_{i=1}^{h/2} \cov(\omega_{q_{v(h,\beta,\Upsilon_1(i))}},\omega_{q_{v(h,\beta,\Upsilon_2(i))}}) \\
&= s_h x \bigg( \sum_{i=0}^{\rho(X)} \sum_{j=0}^{\rho(X)} Q_i \big( \mu(\omega_{q_0}),\dots,\mu(\omega_{q_{\rho(X)}}) \big) Q_j\big( \mu(\omega_{q_0}),\dots,\mu(\omega_{q_{\rho(X)}}) \big) \cov(\omega_{q_i},\omega_{q_j}) \bigg)^{h/2} \\
&= s_h x \bigg(C (\log\log x)^3 + O\bigg(\frac{(\log\log x)^{5/2}}{\log\log\log x}\bigg)\bigg)^{h/2} \\
&= C^{h/2} s_h x (\log\log x)^{3h/2} + O(x(\log\log x)^{3(h - 1)/2})
\end{align*}
by Lemma~\ref{OOM of sum of partial derivs}.
Combining this evaluation with equations~\eqref{other hand} and~\eqref{for this main term} yields
\[
\lim_{x \to \infty} \frac{M_h(x)}{C^{h/2}x(\log\log x)^{3h/2}} = s_h = \frac{h!}{(h/2)!2^{h/2}},
\]
which completes the proof when $h$ is even.
\end{proof}

\section{The method of moments}\label{flourish}

We now describe the argument that deduces the Erd\H os--Kac law for $\log G(n)$ (Theorem~\ref{main theorem}) from the asymptotic formula for the moments given in Proposition~\ref{moment asymptotics}. While this deduction is fairly standard, for the sake of completeness we include the rest of the proof.

For any real number $u$ and positive real number $x$, let $k_x(u)$ denote the number of integers $n \leq x$ such that
$P_n(x) < D(x) + u \cdot \sqrt{C}(\log\log x)^{3/2}$.
Then $\sigma_x(u) = k_x(u)/x$ is the cumulative distribution function of the random variable $Y_x$ obtained by choosing $n\le x$ uniformly at random and then calculating $(P_n(x) - D(x))/\sqrt{C}(\log\log x)^{3/2}$; the $h$th moment of this random variable equals
\[
\int_{-\infty}^\infty u^h \,d\sigma_x(u) = \frac{1}{x} \sum_{n \leq x} \bigg( \frac{P_n(x) - D(x)}{\sqrt C(\log\log x)^{3/2}} \bigg)^h = \frac{M_h(x)}{x C^{h/2} ( \log\log x )^{3h/2}}.
\]
For every fixed $h$, by Proposition~\ref{moment asymptotics}, this $h$th moment converges (as $x\to\infty$) to $s_h = h!/2^{h/2}(\frac h2)!$ when $h$ is even and to $0$ when $h$ is odd. By the ``method of moments'' from probability (see~\cite[Theorem 30.2]{bil86}), the sequence $\{Y_x\}$ of random variables converges in distribution to the unique random variable with these moments, which is the standard normal random variable. (This result is due to Chebyshev for the normal distribution and was later generalized to any random variable that is uniquely determined by its moments.)
In other words, for any real number $u$,
\begin{align}\label{erdos kac for Q}
\lim_{x \to \infty} \frac{1}{x} \#\bigg\{n \leq x \colon \frac{P_n(x) - D(x)}{\sqrt C(\log\log x)^{3/2}} < u \bigg\} = \frac{1}{\sqrt{2\pi}} \int_{-\infty}^u e^{-t^2/2} \, dt.
\end{align}

On the other hand, Propositions~\ref{log G as a polynomial intro version} and~\ref{constant in mean} imply that
\begin{equation} \label{sec7eq}
\frac{P_n(x) - D(x)}{\sqrt C(\log\log x)^{3/2}} = \frac{\log G(n) - A(\log\log x)^2}{\sqrt C(\log\log x)^{3/2}} + O\bigg( \frac1{\log\log\log x} \bigg)
\end{equation}
for all but $O(x/\log\log\log x)$ integers $n \leq x$. Furthermore, $\log\log x = (\log\log n) \big(1 + O(1/\log\log x) \big)$ when $n>x/\log\log\log x$, and therefore we may modify equation~\eqref{sec7eq} to
\begin{equation*}
\frac{P_n(x) - D(x)}{\sqrt C(\log\log x)^{3/2}} = \frac{\log G(n) - A(\log\log n)^2}{\sqrt C(\log\log n)^{3/2}} + O\bigg( \frac1{\log\log\log x} \bigg)
\end{equation*}
for all but $O(x/\log\log\log x)$ integers $n \leq x$.
It follows from this estimate that we also have
\begin{equation}  \label{erdos kac for realz}
\lim_{x \to \infty} \frac{1}{x} \#\bigg\{n \leq x \colon \frac{\log G(n) - A(\log\log n)^2}{\sqrt C(\log\log n)^{3/2}} < u \bigg\} = \frac{1}{\sqrt{2\pi}} \int_{-\infty}^u e^{-t^2/2} \, dt
\end{equation}
(by bounding, for a given real number $u$, the left-hand side of equation~\eqref{erdos kac for realz} above and below by the left-hand side of equation~\eqref{erdos kac for Q} with $u$ replaced, respectively, by $u+\ep$ and $u-\ep$), which is equivalent to the conclusion of Theorem~\ref{main theorem}.

\section{Counting subgroups up to isomorphism, and maximal orders}\label{max order section}

Recall that $I(n)$ denotes the number of isomorphism classes of subgroups of $\Znt$. We are able to quickly establish an Erd{\H o}s--Kac law for $I(n)$ (Theorem~\ref{logIn EK thm}) by relating $\log I(n)$ to two other $\phi$-additive functions that have already been analyzed by Erd{\H o}s and Pomerance.

\begin{lemma}  \label{log In bounds lemma}
For any positive integer $n$, we have $\omega(\phi(n)) \log 2 \leq \log I(n) \leq \Omega(\phi(n)) \log 2$.
\end{lemma}

\begin{proof}
Let $I_p(n)$ denote the number of isomorphism classes of $p$-subgroups of $\Znt$; as we saw with $G(n)$, we again have
\[
\log I(n) = \sum_{p \mid \phi(n)} \log I_p(n).
\]
This is already enough to imply the lower bound $\omega(\phi(n)) \log 2 \leq \log I(n)$: for every prime $p\mid\phi(n)$, the quantity $I_p(n)$ counts at least two subgroups of $\Znt$, namely the Sylow $p$-subgroup and the trivial subgroup.

For any such prime $p$, write the Sylow $p$-subgroup of $\Znt$ as $\Z_{p^{\alpha_1}} \times \Z_{p^{\alpha_2}} \times \cdots$ for some partition $\balpha$ of the integer~$\nu_p(\phi(n))$. Then $I_p(n)$ is exactly the number of subpartitions of~$\balpha$. Certain subsets of the Ferrers diagram corresponding to $\balpha$ correspond to subpartitions, while many subsets do not; but the total number of subsets, $2^{\nu_p(\phi(n))}$, is certainly an upper bound for the number of subpartitions. We conclude that
\begin{equation}  \label{wasteful?}
\sum_{p \mid \phi(n)} \log I_p(n) \leq \sum_{p \mid \phi(n)} \log \big( 2^{\nu_p(\phi(n))} \big) = \sum_{p \mid \phi(n)} {\nu_p(\phi(n))} \log2 = \Omega(\phi(n)) \log 2,
\end{equation}
which is the desired upper bound.
\end{proof}

\begin{proof}[Proof of Theorem~\ref{logIn EK thm}]
Erd{\H o}s and Pomerance~\cite{ep85} have shown that both $\omega(\phi(n))$ and $\Omega(\phi(n))$ satisfy Erd{\H o}s--Kac laws, in both cases with mean $\frac{1}{2} (\log\log n)^2$ and variance is $\frac{1}{3} (\log\log n)^3$. Thus, as a consequence of Lemma~\ref{log In bounds lemma}, $\log I(n)$ satisfies an Erd{\H o}s--Kac law with mean $\frac{\log 2}{2} (\log\log n)^2$ and variance $\frac{\log 2}{3} (\log\log n)^3$ as well.
\end{proof}

It might be surprising that the simple bounds from Lemma~\ref{log In bounds lemma} suffice to establish this Erd{\H o}s--Kac law, despite how seemingly wasteful the inequality~\eqref{wasteful?} is. We view this as a reflection of the anatomical fact that typically, most primes dividing $\phi(n)$ are large and most large primes dividing $\phi(n)$ do so only to the first power.

We turn now to the question of determining how large the values of $G(n)$ and $I(n)$ can become. We start with a pair of arguments (an upper bound and a construction) that together show that the maximal order of $\log G(n)$ has order of magnitude ${(\log x)^2/\log\log x}$.
In both arguments, it will be helpful to observe that
\[
\lambda_p(n) = \sum_{j=1}^{\lambda_p(n)} 1 \le \sum_{j=1}^{\lambda_p(n)} \ovomega_{p^j}(n)
\]
for any prime $p\mid\phi(n)$, and therefore
\begin{equation}  \label{first constraint}
\sum_{p\mid\phi(n)} \lambda_p(n) \log p \le \sum_{p\mid\phi(n)} \sum_{j=1}^{\lambda_p(n)} \ovomega_{p^j}(n) \log p = \log\phi(n) < \log x
\end{equation}
by equation~\eqref{sum of its parts}.

\begin{proof}[Proof of the upper bound in Theorem~\ref{max order Gn}]
Proposition~\ref{gpn as a polynomial} gives
\begin{align}
\log G(n) = \sum_{p \mid \phi(n)} \log G_p(n) &= \sum_{p \mid \phi(n)} \bigg( \frac{\log p}4 \sum_{j=1}^{\lambda_p(n)} \ovomega_{p^j}(n)^2 + O(\lambda_p(n) \log p) \bigg) \notag \\
&= \frac14 \sum_{p \mid \phi(n)} \sum_{j=1}^{\lambda_p(n)} \ovomega_{p^j}(n)^2 \log p + O\bigg( \sum_{p \mid \phi(n)} \lambda_p(n) \log p \bigg) \notag \\
&= \frac14 \sum_{p \mid \phi(n)} \sum_{j=1}^{\lambda_p(n)} \ovomega_{p^j}(n)^2 \log p + O(\log x)
\label{about to pull one out}
\end{align}
by equation~\eqref{first constraint}. On the other hand,
\begin{equation}  \label{classical omega bound}
\ovomega_{p^j}(n) \leq \omega_{p^j}(n) + 2 \leq \omega(n) + 2 < \frac{\log x}{\log\log x} \bigg( 1 + \frac1{\log\log x} \bigg)
\end{equation}
by the classical upper bound for $\omega(n)$~\cite[Theorem 2.10]{mv07}. We use this bound on one of the two factors of $\ovomega_{p^j}(n)$ in each summand on the right-hand side of equation~\eqref{about to pull one out}, obtaining
\begin{align*}
\log G(n) &< \frac14 \sum_{p \mid \phi(n)} \sum_{j=1}^{\lambda_p(n)} \ovomega_{p^j}(n) \log p \cdot \frac{\log x}{\log\log x} \bigg( 1 + \frac1{\log\log x} \bigg) + O(\log x) \\
&= \frac{\log x}{4\log\log x} \bigg( 1 + \frac1{\log\log x} \bigg) \log \phi(n) + O(\log x) < \frac{(\log x)^2}{4\log\log x} \bigg( 1 + \frac1{\log\log x} \bigg)
\label{about to pull one out}
\end{align*}
by equation~\eqref{first constraint} again.
\end{proof}

\begin{proof}[Proof of the lower bound in Theorem~\ref{max order Gn}]
Choose $B=B(3)$ so that Theorem~\ref{bombvino} is valid, and set
\begin{align*}
V &= \frac{(\log x)^2}{(\log\log x)^{2B+1}} \bigg( 1 - \frac1{\log\log x} \bigg) \\
Q &= \frac{\log x}{(\log\log x)^{2B+1}};
\end{align*}
note that $Q < V^{1/2}/(\log V)^B$ when $x$ is large enough.
Thus by equation~\eqref{theta BV},
\[
\sum_{Q < p < 2Q} \bigg|\theta(V; p, 1) - \frac V{p-1}\bigg| \ll \frac{V}{(\log V)^3}.
\]
Since the number of primes between $Q$ and $2Q$ is $\gg Q/\log Q$, we may choose a prime $p$ in that interval such that
\begin{align}
\theta(V, p, 1) &= \frac V{p-1} + O\bigg( \frac{V}{(\log V)^3} \frac{\log Q}Q \bigg) \notag \\
&\le \frac VQ + O\bigg( \frac{V}{(\log V)^3} \frac{\log Q}Q \bigg) = \log x - \frac{\log x}{\log\log x} + O\bigg( \frac{\log x}{(\log\log x)^2} \bigg).
\label{less than log x}
\end{align}
Now, fixing this prime $p$ that was chosen above, consider the integer
\[
n = \prod_{\substack{q \le V \\ q \equiv 1 \mod p}} q = e^{\theta(V, p, 1)},
\]
where $p$ is the prime chosen above; by equation~\eqref{less than log x}, we see that $n<x$ when $x$ is sufficiently large. Notice also that $\log p = \log\log x + O(\log\log\log x)$ and $\log V = 2\log\log x + O(\log\log\log x)$, and therefore
\[
\omega_p(n) = \pi(V;p,1) \ge \frac{\theta(V;p,1)}{\log V} \ge \frac{\log x}{2\log\log x} \bigg( 1 + O\bigg( \frac{\log\log\log x}{\log\log x} \bigg) \bigg)
\]
by equation~\eqref{less than log x}. Consequently,
\begin{align*}
\log G(n) &\ge \log G_p(n) \\
&= \frac{\log p}{4} \sum_{j = 1}^{\lambda_p(n)} \ovomega_{p^j}(n)^2 + O(\lambda_p(n) \log p) \\
&\ge \frac{\log p}{4} \omega_{p}(n)^2 + O(\lambda_p(n) \log p) \\
&\ge \frac{\log\log x + O(\log\log\log x)}{4} \bigg( \frac{\log x}{2\log\log x} \bigg)^2 \bigg( 1 + O\bigg( \frac{\log\log\log x}{\log\log x} \bigg) \bigg) + O(\lambda_p(n) \log p) \\
&= \frac{\log^2 x}{16\log\log x} \bigg( 1 + O\bigg( \frac{\log\log\log x}{\log\log x} \bigg) \bigg) + O(\lambda_p(n) \log\log x).
\end{align*}
Since equation~\eqref{first constraint} implies that $\lambda_p(n) < \log x$, the above estimate establishes the desired lower bound.
\end{proof}

We believe that the upper bound (with leading constant $\frac14$) gives the true asymptotic size of the maximal order of $\log G(n)$; in particular, if one assumes the Elliott--Halberstam conjecture, then the construction giving the lower bound can easily be modified to produce a leading constant $\frac14$ instead of the current~$\frac1{16}$.

Theorem~\ref{max order Gn} shows that the maximal order of $\log G(n)$ is substantially larger than the typical size of $\log G(n)$. The same phenomenon holds, to a somewhat lesser degree, for $\log I(n)$, which we now show via another pair of arguments (an upper bound and a construction) after the following preliminary lemma.

\begin{lemma}  \label{log recip phi sum}
For any $x\ge3$ and any integer $m\le x$,
\[
\sum_{p \mid m} \frac{1}{\log p} < \frac{\log x}{(\log\log x)^2} + O\bigg( \frac{\log x}{(\log\log x)^3} \bigg).
\]
\end{lemma}

\begin{proof}
First suppose that $m_0 = \prod_{p\le y} p$ for some real number~$y$. Then $\log x > \log m_0 = \sum_{p\le y} \log p = \theta(y) = y + O(y/\log y)$ by the prime number theorem, which implies that $y < \log x + O(\log x/\log\log x)$. Then, by partial summation,
\[
\sum_{p \mid m_0} \frac{1}{\log p} = \sum_{p \le y} \frac{1}{\log p} = \frac{y}{\log^2 y} < \frac{\log x}{(\log\log x)^2} + O\bigg( \frac{\log x}{(\log\log x)^3} \bigg).
\]

For general $m$, choose a real number $y$ such that $\pi(y) = \omega(m)$, and define $m_0 = \prod_{p\le y} p$. Then $m_0 \le m \le x$, while $\sum_{p \mid m_0} 1/\log p \ge \sum_{p \mid m} 1/\log p$ since both sums have the same number of summands and each individual summand in the first sum is at least as large as the corresponding summand in the second sum. Consequently, the desired upper bound for $\sum_{p \mid m} 1/\log p$ follows from the already established upper bound for $\sum_{p \mid m_0} 1/\log p$.
\end{proof}

\begin{proof}[Proof of the upper bound in Theorem~\ref{max order In}]
Given $n \leq x$, we write
\[
\phi(n) = \prod_{p^{k_p} \parallel \phi(n)} p^{k_p};
\]
since $\phi(n) \le n \leq x$,
\begin{equation}  \label{Iup constraint}
\sum_{p^{k_p} \parallel \phi(n)} k_p \log p = \log\phi(n) \le \log x.
\end{equation}
For any prime $p$ dividing $\phi(n)$, the $p$-Sylow subgroup of $\Znt$ is of the form $\Z_{p^{\alpha_1}} \times \Z_{p^{\alpha_2}} \times \cdots$ for some partition $\balpha$ of the integer~$k_p$. As discussed at the beginning of Section~\ref{subgroups of p groups section}, $I_p(n)$ is the number of subpartitions of $\balpha$. We bound this number crudely by noting that every subpartition of $\balpha$ is a partition of some integer $j \in \{0,1,\dots,k_p\}$; therefore, with $P(m)$ denoting the usual partition function,
\begin{equation}  \label{partitions arise again}
I_p(n) \leq \sum_{j = 0}^{k_p} P(j) \leq (k_p + 1) P(k_p).
\end{equation}
A consequence of Lehmer's formula for the partition function, as described in~\cite[proof of Theorem 2.1]{BO}, is that for every positive integer~$m$,
\[
(m+1)P(m) < (m+1)\cdot \frac{\sqrt3}{12m} \bigg( 1+\frac1{\sqrt m} \bigg) \exp\bigg( \frac\pi6\sqrt{24m-1} \bigg) < \exp\bigg( \pi\sqrt{\frac23} m \bigg),
\]
where the last inequality can be verified by an easy calculation. In particular, the upper bound~\eqref{partitions arise again} implies that $\log I_p(n) < \pi\sqrt{2k_p/3}$, and thus
\begin{equation}  \label{pre CS}
\log I(n) = \sum_{p \mid \phi(n)} \log I_p(n) < \pi\sqrt{\frac23} \sum_{p \mid \phi(n)} \sqrt{k_p}.
\end{equation}
But by Cauchy--Schwarz,
\[
\bigg( \sum_{p \mid \phi(n)} \sqrt{k_p} \bigg)^2 \le \bigg( \sum_{p \mid \phi(n)} k_p \log p \bigg) \bigg( \sum_{p \mid \phi(n)} \frac{1}{\log p} \bigg) \le \log x \cdot \frac{\log x}{(\log\log x)^2} \bigg( 1 + O\bigg( \frac1{\log\log x} \bigg) \bigg)
\]
by equation~\eqref{Iup constraint} and Lemma~\ref{log recip phi sum}; combining this bound with equation~\eqref{pre CS} completes the proof of the upper bound.
\end{proof}

\begin{proof}[Proof of the lower bound in Theorem~\ref{max order In}]
We employ a strategy suggested by Pomerance (private communication). Set $U = \frac15\log x - \log\log x$, define $m = \prod_{p \le U} p$, and let $q$ be the smallest prime that is congruent to $1\mod m$. By Linnik's theorem, with the best current value of Linnik's constant due to Xylouris~\cite{tri11}, we know that $q \ll m^5$. On the other hand, by the prime number theorem,
\[
\log m = \theta(U) = U + O\bigg( \frac U{\log^2 U} \bigg) = \frac15\log x - \log\log x + O\bigg( \frac{\log x}{(\log\log x)^2} \bigg),
\]
which shows that $m=o(x^{1/5})$ and therefore $q<x$ when $x$ is large enough.

Since $m$ divides $q-1$, the prime number theorem also gives
\begin{align*}
\omega(\phi(q)) = \omega(q - 1) \ge \omega(m) = \pi(U) &= \frac U{\log U} + O\bigg( \frac U{\log^2 U} \bigg) \\
&= \frac{\log x}{5\log\log x} + O\bigg( \frac{\log x}{(\log\log x)^2} \bigg).
\end{align*}
The lower bound now follows from inequality $\log I(q) \geq \log 2 \cdot \omega(\phi(q))$, as noted in the proof of Lemma~\ref{log In bounds lemma}.
\end{proof}

Note that the constant $\frac15\log 2$ can be improved to any number less than $\log 2$ if one is willing to assume Montgomery's conjecture on the error term in the prime number theorem for arithmetic progressions (as stated by Friedlander and Granville~\cite[conjecture 1(b)]{FG}). However, even this assumption is not enough to close the gap between the constants in the upper and lower bounds (note that $\log2 \approx 0.69315$ while $\pi\sqrt{2/3} \approx 2.56651$).

\section*{Acknowledgements}

The authors thank Carl Pomerance for helpful conversations concerning some proofs in this paper.
The authors were supported in part by a National Sciences and Engineering Research Council of Canada Discovery Grant.

\bibliographystyle{amsplain}
\bibliography{refs}

\providecommand{\bysame}{\leavevmode\hbox to3em{\hrulefill}\thinspace}
\providecommand{\MR}{\relax\ifhmode\unskip\space\fi MR }
\providecommand{\MRhref}[2]{%
  \href{http://www.ams.org/mathscinet-getitem?mr=#1}{#2}
}
\providecommand{\href}[2]{#2}
\begin{thebibliography}{10}

\bibitem{AH}
A.~Akbary and K.~Hambrook, \emph{A variant of the {B}ombieri--{V}inogradov
  theorem with explicit constants and applications}, Math. Comp. \textbf{84}
  (2015), no.~294, 1901--1932.

\bibitem{BO}
C.~Bessenrodt and K.~Ono, \emph{Maximal multiplicative properties of
  partitions}, Ann. Comb. \textbf{20} (2016), no.~1, 59--64.

\bibitem{bil86}
P.~Billingsley, \emph{Probability and measure}, 3rd ed., Wiley Series in
  Probability and Mathematical Statistics, John Wiley \& Sons, Inc., New York,
  1995, A Wiley--Interscience Publication.

\bibitem{ek40}
P.~Erd{\H o}s and M.~Kac, \emph{The {G}aussian law of errors in the theory of
  additive number theoretic functions}, American Journal of Mathematics
  \textbf{62} (1940), no.~1/4, 343--352.

\bibitem{Erd1978-1979}
P.~Erd{\H o}s and J-L. Nicolas, \emph{Sur la fonction ``nombre de facteurs
  premiers de n''}, S{\' e}minaire Delange--Pisot--Poitou. Th{\' e}orie des
  nombres \textbf{20} (1978--1979), no.~2, 1--19.

\bibitem{ep85}
P.~Erd{\H o}s and C.~Pomerance, \emph{The normal number of prime factors of
  $\varphi(n)$}, Rocky Mtn. J. Math. \textbf{15} (1985), 343--352.

\bibitem{FG}
J.~Friedlander and A.~Granville, \emph{Limitations to the equi-distribution of
  primes. {I}}, Ann. of Math. (2) \textbf{129} (1989), no.~2, 363--382.

\bibitem{gs07}
A.~Granville and K.~Soundararajan, \emph{Sieving and the {E}rd{\H o}s--{K}ac
  theorem}, Equidistribution in number theory, an introduction, NATO Sci. Ser.
  II Math. Phys. Chem., vol. 237, Springer, Dordrecht, 2007, pp.~15--27.

\bibitem{ik04}
H.~Iwaniec and E.~Kowalski, \emph{Analytic number theory}, American
  Mathematical Society Colloquium Publications, vol.~53, American Mathematical
  Society, Providence, RI, 2004.

\bibitem{mv07}
H.L. Montgomery and R.C. Vaughan, \emph{Multiplicative number theory. {I}.
  {C}lassical theory}, Cambridge Studies in Advanced Mathematics, vol.~97,
  Cambridge University Press, Cambridge, 2007.

\bibitem{nor}
K.K. Norton, \emph{On the number of restricted prime factors of an integer.
  {I}}, Illinois J. Math. \textbf{20} (1976), no.~4, 681--705.

\bibitem{pom}
C.~Pomerance, \emph{On the distribution of amicable numbers}, J. Reine Angew.
  Math. \textbf{293/294} (1977), 217--222.

\bibitem{ste92}
T.~Stehling, \emph{On computing the number of subgroups of a finite abelian
  group}, Combinatorica \textbf{12} (1992), no.~4, 475--479.

\bibitem{tri11}
T.~Xylouris, \emph{\"{U}ber die {N}ullstellen der {D}irichletschen
  {L}-{F}unktionen und die kleinste {P}rimzahl in einer arithmetischen
  {P}rogression}, Bonner Mathematische Schriften [Bonn Mathematical
  Publications], vol. 404, Universit\"at Bonn, Mathematisches Institut, Bonn,
  2011, Dissertation for the degree of Doctor of Mathematics and Natural
  Sciences at the University of Bonn, Bonn, 2011.

\end{thebibliography}

\end{document}